\documentclass{amsart}
\usepackage{amssymb}
\usepackage{amsmath}
\usepackage{mathrsfs}
\usepackage{graphicx,epic,eepic}
\usepackage[usenames]{color}
\usepackage[all]{xy}

\title[Eight points in projective space]
  {The geometry of eight points in projective space:  Representation theory, Lie theory, dualities}

\author[Benjamin Howard, John Millson, Andrew Snowden and Ravi Vakil]
{Benjamin Howard, John Millson, Andrew Snowden and Ravi Vakil*}

\thanks{*B.~Howard was partially supported by NSF grants DMS-0405606 and
DMS-0703674.  J.~Millson was partially supported by the Simons Foundation
and NSF grants DMS-0405606 and DMS-05544254.  A.~Snowden was partially
supported by NSF fellowship DMS-0902661.  R.~Vakil was partially supported
by NSF grant DMS-0801196.}

\date{March 27, 2011.}

\newcommand{\cut}[1]{}

\newtheorem{theorem}{Theorem}[section]
\newtheorem{lemma}[theorem]{Lemma}
\newtheorem{corollary}[theorem]{Corollary}
\newtheorem{proposition}[theorem]{Proposition}

\newtheorem{problem}[theorem]{Problem}
\newtheorem{conj}[theorem]{Conjecture}

\theoremstyle{definition}

\theoremstyle{remark}
\newtheorem{remark}[theorem]{\it Remark}

\DeclareMathOperator{\Sing}{Sing}
\DeclareMathOperator{\Proj}{Proj}
\DeclareMathOperator{\Sym}{Sym}
\DeclareMathOperator{\bw}{{\bigwedge}^{\! 2}}

\DeclareMathOperator{\Hom}{Hom}
\DeclareMathOperator{\Aut}{Aut}
\DeclareMathOperator{\End}{End}
\DeclareMathOperator{\sgn}{sgn}
\DeclareMathOperator{\Tor}{Tor}
\DeclareMathOperator{\Spec}{Spec}

\DeclareMathOperator{\Pic}{Pic}
\DeclareMathOperator{\Sec}{Sec}
\DeclareMathOperator{\Cone}{Cone}
\DeclareMathOperator{\codim}{codim}
\DeclareMathOperator{\Ext}{Ext}

\def\GL{\mathrm{GL}}
\def\SL{\mathrm{SL}}

\def\SS{\mathfrak{S}}

\def\fieldk{\mathbf{k}}
\def\Z{\mathbb{Z}}

\def\P{\mathbb{P}}
\def\proj{\mathbb{P}}
\def\C{\mathbb{C}}
\def\Q{\mathbb{Q}}
\def\F{\mathbb{F}}

\def\wt{\mathbf{wt}}
\def\cq{/\!/}

\def\cubic{\mathcal{C}}
\def\oh{\mathcal{O}}
\def\quintic{\mathcal{Q}}
\def\septic{\mathcal{S}}
\def\hessian{\mathcal{H}}

\let\mf\mathfrak
\let\mc\mathcal

\let\wt\widetilde
\let\ol\overline

\let\bs\backslash

\begin{document}


\begin{abstract}
This paper deals with the geometry of the space (GIT quotient) $M_8$ of $8$
  points in $\proj^1$, and the Gale-quotient $N'_8$ of the GIT quotient
  of $8$ points in $\proj^3$.     

  The space $M_8$ comes with a natural embedding in $\proj^{13}$, or
  more precisely, the projectivization of the $\SS_8$-representation
  $V_{4,4}$.  There is a single $\SS_8$-skew cubic $\cubic$ in
  $\proj^{13}$.  The fact that $M_8$ lies on the skew cubic $\cubic$ is a consequence
  of Thomae's formula for hyperelliptic curves, but more is true:
  $M_8$ is the singular locus of $\cubic$.    These
  constructions yield the free resolution of $M_8$, and are used in
  the determination of the ``single'' equation cutting out the GIT
  quotient of $n$ points in $\proj^1$ in general \cite{hmsv2}.

The space $N'_8$ comes with a natural embedding in $\proj^{13}$,  or
more precisely, $\proj V_{2,2,2,2}$.  There is a single
skew quintic $\quintic$ containing $N'_8$, and $N'_8$ is the singular
locus of the skew quintic $\quintic$.

The skew cubic $\cubic$ and skew quintic $\quintic$ are projectively
dual.  (In particular, they are surprisingly singular, in the sense of
having a dual of remarkably low degree.)  The divisor on the skew
cubic blown down by the dual map is the secant variety $\Sec(M_8)$,
and the contraction $\Sec(M_8) \dashrightarrow N'_8$ factors through
$N_8$ via the space of $8$ points on a quadric surface.  We conjecture
(Conjecture~\ref{conj}) that the divisor on the skew quintic blown
down by the dual map is the quadrisecant variety of $N'_8$ (the
closure of the union of quadrisecant {\em lines}), and that the
quintic $\quintic$ is the trisecant variety.  The resulting picture
extends the classical duality in the 6-point case between the Segre
cubic threefold and the Igusa quartic threefold.

We note
  that there are a number of geometrically natural varieties that are
  (related to) the singular loci of remarkably singular cubic
  hypersurfaces, e.g.  \cite{ch}, \cite{beauville}, etc.
\end{abstract}

\maketitle

\tableofcontents

\section{Introduction}

This note discusses the  geometry of the spaces
\begin{displaymath}
M_8 := (\proj^1)^8 \!\cq \SL(2), \quad  N_8 := (\proj^3)^8 \!\cq \SL(4) \quad \text{and}
\quad (\proj^5)^8 \!\cq \SL(6),
\end{displaymath}
each the GIT quotient of $8$ points in projective space with respect to the ``usual'' linearization
$\mc{O}(1,\dots,1)$.  For each of these quotients $Q$, let $R_{\bullet}(Q)$  be the corresponding (graded) ring of
invariants.  (Coble\nobreakdash\textendash)Gale duality gives a canonical isomorphism between the first and third, via
a canonical isomorphism of the graded rings of invariants.  Gale duality gives an involution on $N_8$, through an
involution of its underling graded ring $R_{\bullet}(N_8)$.   (An explicit description of Gale duality in terms of
tableaux due to \cite{hm} is given in the proof of Proposition~\ref{anotherbrowncow}.)  Our goal is to study and relate
$M_8$ and $N_8$ (and the Gale-quotient $N'_8$ of $N_8$) and their extrinsic geometry.  The key constructions are dual
hypersurfaces $\cubic$ and $\quintic$ in $\proj^{13}$ of degrees three and five respectively; for example, $M_8 =
\Sing(\cubic)$ and $N'_8 = \Sing(\quintic)$ (\S \ref{s:singQ}).  The partial derivatives of $\cubic$ (we sloppily 
identify hypersurfaces and their underlying equations), which cut out $M_8$, will be referred to as ``the $14$
quadratic relations''; they span an irreducible $\SS_8$-representation (of type $2+2+2+2$,
see Proposition~\ref{p:repfacts}),
and up to symmetry there is only one quadratic relation (given in appropriate coordinates by a simple binomial relation
\eqref{e:simplebinomial}).


In \cite{hmsv2}, we give {\em all} relations among generators of the graded rings for $(\proj^1)^n \cq \SL(2)$, with
{\em any} linearization.  In each case the graded rings are generated in one degree, so the quotients come with a
natural projective embedding.  The general case reduces to the linearizations $1^n$, with $n$ even.  In this $1^n$ case,
with the single exception of $n=6$, there is (up to $\SS_n$-symmetry) a single quadratic equation, which is binomial in
the Kempe generators (Specht polynomials).  The quadratic for the case $n \geq 8$ is pulled back from the (unique up to
symmetry) $n=8$ quadratic discussed below, which forms the base case of an induction.  We  indicate in \S
\ref{s:relationtosix} how the only case of smaller $n$ with interesting geometry ($n=6$, related to Gale duality, and
projective duality of the Segre cubic and the Igusa quartic) is also visible in  the boundary of the structure we
describe here.  Thus various beautiful structures of GIT quotients of $n$ points on $\proj^1$ are all consequences of
the  geometry in the $8$-point space $M_8$ discussed in this paper. 

The main results are outlined in \S \ref{mainconstructions}, and come in three logically independent parts.  The first
deals with the relationship between $M_8$ and $N_8$.  The second deals solely with $M_8$, and the third with $N_8$.
\begin{enumerate}
\item[(A)] In \S \ref{s:dualities}, we describe the intricate relationship between $M_8$ and $N_8$, 
summarized in Figure~\ref{f:figureeight}.  We note that this section does not use the fact that the ideal cutting
out $M_8$ is {\em generated} by the 14 quadratic relations (established by pure thought in (B)), only that it lies in the
{\em intersection} of the 14 quadratic relations (\S \ref{s:threeone}).  We also do not use that $N'_8$ is the Gale-quotient
of $N_8$ (established in (C)), only that $N'_8$ corresponds to the subring of the ring of invariants of $(\proj^3)^8
\cq \SL(4)$ generated in degree $1$.  (For this reason we take this as
our initial definition of $N'_8$.)
\item[(B)] In \S \ref{s:syzygy1}--\ref{s:syzygy3}, we give a Lie-theoretic proof of the fact
that there are no linear syzygies among the 14 quadratic relations (Theorem~\ref{prop:nosyz}).  This (in combination
with results of \cite{hmsv2}) gives a pure thought proof that the 14 quadratic relations generate the ideal of $M_8$,
the base case of the main induction of \cite{hmsv2} (Corollary~\ref{prop:i2-i3}).  This was known earlier via computer
calculation by a number of authors (Maclagan, private communication; Koike \cite{koike}; and Freitag and Salvati-Manni
\cite{fs}), but we wished to show the structural reasons for this result in order to make the main theorem of
\cite{hmsv2} (giving all relations for all GIT quotients of $(\proj^1)^n\cq \SL(2)$) computer-independent. (Strictly
speaking, in \cite{hmsv2}, computers were used to deal with the character theory of small-dimensional
$\SS_6$-representations, but this could certainly be done by hand with some effort.)  In \S \ref{s:freeresolution}, we
use the absence of linear syzygies to determine the graded free resolution of (the ring of invariants of) $M_8$.
\item[(C)] In the short concluding section \S \ref{s:N8}, we verify
  with the aid of a computer that the subring $R_{\bullet}(N'_8)$ of
  $R_{\bullet}(N_8)$ generated in degree $1$ is indeed the ring of
  Gale invariants, and that the skew quintic is the {\em only} skew quintic
  relation in both $N'_8$ and $N_8$.  This is done by verifying that
  $R_{\bullet}(N_8)$ is generated in degrees one and two, and
  determining the actions of $\SS_8$ and Gale duality on these
  generating sets.
\end{enumerate}
To be clear on the use of computer calculation (as opposed to pure thought): in \S \ref{s:dualities}, we use a computer
only to intersect two curves in $\proj^2$; in \S \ref{s:syzygy1}--\ref{s:freeresolution}, computers are not used; and
computer calculation is central to \S \ref{s:N8}.

We describe other manifestations of the ring of invariants of $M_8$ in \S \ref{othermanifestations}.  Miscellaneous
algebraic results about $M_8$ that may be useful to others are given in \S \ref{s:miscalgebra}.  We sketch how the
beautiful classical geometry of the six point case is visible at the boundary in \S \ref{s:relationtosix}.  The
justifications of the statements made in \S \ref{mainconstructions} are given in the rest of the paper.

\subsection{Notation}
{\em In general, we work over a  field $\fieldk$ of characteristic $0$.}  Most statements work away from a known
finite list of primes, so we occasionally give characteristic-specific statements.  For a partition $\lambda$ of $n$,
we write $V_{\lambda}$ for the corresponding irreducible representation of $\SS_n$.  The $\SS_8$-representations
important for us are the trivial ($V_8$) and sign ($\sgn := V_{1^8}$) representations, and the two $14$-dimensional
representations $V_{4,4}$ and  $V_{2,2,2,2}$.  The latter two are
skew-dual: $V_{4,4} \otimes \sgn \cong V_{2,2,2,2}$.  The
representation $V_{3,1,1,1,1,1}$ appears in \S \ref{s:N8}.

\subsection{Main constructions (see Figure~\ref{f:figureeight}) }
\label{mainconstructions}

All statements made here will be justified later in the paper.
\begin{figure}[ht]
\begin{center}
\setlength{\unitlength}{0.00083333in}
\begingroup\makeatletter\ifx\SetFigFont\undefined%
\gdef\SetFigFont#1#2#3#4#5{%
  \reset@font\fontsize{#1}{#2pt}%
  \fontfamily{#3}\fontseries{#4}\fontshape{#5}%
  \selectfont}%
\fi\endgroup%
{\renewcommand{\dashlinestretch}{30}
\begin{picture}(5772,2947)(0,-10)
\put(5409,2272){\makebox(0,0)[lb]{{\SetFigFont{8}{9.6}{\rmdefault}{\mddefault}{\updefault}$\subset$}}}
\blacken\path(1164.000,2227.000)(1134.000,2347.000)(1104.000,2227.000)(1164.000,2227.000)
\path(1134,2347)(1134,2047)
\blacken\path(1104.000,2167.000)(1134.000,2047.000)(1164.000,2167.000)(1104.000,2167.000)
\put(459,1897){\makebox(0,0)[lb]{{\SetFigFont{8}{9.6}{\rmdefault}{\mddefault}{\updefault}$(\proj^1)^8// \Aut \proj^1$}}}
\put(684,97){\makebox(0,0)[lb]{{\SetFigFont{8}{9.6}{\rmdefault}{\mddefault}{\updefault}$= N_8$}}}
\put(459,2197){\makebox(0,0)[lb]{{\SetFigFont{5}{6.0}{\rmdefault}{\mddefault}{\updefault}Gale duality}}}
\put(84,772){\makebox(0,0)[lb]{{\SetFigFont{5}{6.0}{\rmdefault}{\mddefault}{\updefault}Gale duality}}}
\path(1734,2422)(2109,2347)
\blacken\path(1985.447,2341.117)(2109.000,2347.000)(1997.214,2399.951)(1985.447,2341.117)
\path(1734,1972)(2109,2197)
\blacken\path(2021.536,2109.536)(2109.000,2197.000)(1990.666,2160.985)(2021.536,2109.536)
\blacken\path(2015.085,766.502)(2109.000,847.000)(1988.252,820.167)(2015.085,766.502)
\path(2109,847)(1659,622)
\path(1809,547)(3984,772)
\blacken\path(3867.724,729.811)(3984.000,772.000)(3861.550,789.493)(3867.724,729.811)
\dashline{60.000}(3834,997)(3009,2197)
\blacken\path(3101.705,2115.111)(3009.000,2197.000)(3052.262,2081.119)(3101.705,2115.111)
\dashline{60.000}(3909,2197)(3084,997)
\blacken\path(3127.262,1112.881)(3084.000,997.000)(3176.705,1078.889)(3127.262,1112.881)
\dashline{60.000}(3684,2197)(459,547)
\blacken\path(552.166,628.365)(459.000,547.000)(579.494,574.950)(552.166,628.365)
\dashline{60.000}(1884,2722)(1884,1747)
\dashline{60.000}(1884,1222)(1884,247)
\blacken\path(4839.000,2077.000)(4809.000,2197.000)(4779.000,2077.000)(4839.000,2077.000)
\dashline{60.000}(4809,2197)(4809,997)
\blacken\path(4779.000,1117.000)(4809.000,997.000)(4839.000,1117.000)(4779.000,1117.000)
\put(2184,2272){\makebox(0,0)[lb]{{\SetFigFont{8}{9.6}{\rmdefault}{\mddefault}{\updefault}$M_8 = \Sing(\text{cubic})$ }}}
\put(2184,847){\makebox(0,0)[lb]{{\SetFigFont{8}{9.6}{\rmdefault}{\mddefault}{\updefault}$N'_8 = \Sing(\text{quintic})$}}}
\put(2484,2722){\makebox(0,0)[lb]{{\SetFigFont{5}{6.0}{\rmdefault}{\mddefault}{\updefault}$\dim 5$}}}
\put(309,397){\makebox(0,0)[lb]{{\SetFigFont{8}{9.6}{\rmdefault}{\mddefault}{\updefault}$(\proj^3)^8 // \Aut \proj^3$}}}
\put(1359,697){\makebox(0,0)[lb]{{\SetFigFont{5}{6.0}{\rmdefault}{\mddefault}{\updefault}2:1 (Gale)}}}
\put(2334,1447){\makebox(0,0)[lb]{{\SetFigFont{5}{6.0}{\rmdefault}{\mddefault}{\updefault}Segre}}}
\put(2409,1597){\makebox(0,0)[lb]{{\SetFigFont{5}{6.0}{\rmdefault}{\mddefault}{\updefault}$f$}}}
\put(2484,397){\makebox(0,0)[lb]{{\SetFigFont{5}{6.0}{\rmdefault}{\mddefault}{\updefault}$\dim 9$}}}
\put(3459,847){\makebox(0,0)[lb]{{\SetFigFont{8}{9.6}{\rmdefault}{\mddefault}{\updefault}$\subset$}}}
\put(3459,2272){\makebox(0,0)[lb]{{\SetFigFont{8}{9.6}{\rmdefault}{\mddefault}{\updefault}$\subset$}}}
\put(4359,2272){\makebox(0,0)[lb]{{\SetFigFont{8}{9.6}{\rmdefault}{\mddefault}{\updefault}$\subset$}}}
\put(4359,847){\makebox(0,0)[lb]{{\SetFigFont{8}{9.6}{\rmdefault}{\mddefault}{\updefault}$\subset$}}}
\put(3759,847){\makebox(0,0)[lb]{{\SetFigFont{8}{9.6}{\rmdefault}{\mddefault}{\updefault}divisor}}}
\put(3684,2272){\makebox(0,0)[lb]{{\SetFigFont{8}{9.6}{\rmdefault}{\mddefault}{\updefault}$\Sec(M_8)$}}}
\put(5634,847){\makebox(0,0)[lb]{{\SetFigFont{8}{9.6}{\rmdefault}{\mddefault}{\updefault}$\proj^{13 \vee}$}}}
\put(459,2422){\makebox(0,0)[lb]{{\SetFigFont{8}{9.6}{\rmdefault}{\mddefault}{\updefault}$(\proj^5)^8 // \Aut \proj^5$}}}
\put(4209,1672){\makebox(0,0)[lb]{{\SetFigFont{5}{6.0}{\rmdefault}{\mddefault}{\updefault}projective}}}
\put(4284,1522){\makebox(0,0)[lb]{{\SetFigFont{5}{6.0}{\rmdefault}{\mddefault}{\updefault}dual}}}
\put(5634,1597){\makebox(0,0)[lb]{{\SetFigFont{5}{6.0}{\rmdefault}{\mddefault}{\updefault}dual}}}
\put(3084,1297){\makebox(0,0)[lb]{{\SetFigFont{5}{6.0}{\rmdefault}{\mddefault}{\updefault}$f'$}}}
\put(3759,2722){\makebox(0,0)[lb]{{\SetFigFont{5}{6.0}{\rmdefault}{\mddefault}{\updefault}$\dim 11$}}}
\put(3759,397){\makebox(0,0)[lb]{{\SetFigFont{5}{6.0}{\rmdefault}{\mddefault}{\updefault}$\dim 11$}}}
\put(4659,397){\makebox(0,0)[lb]{{\SetFigFont{5}{6.0}{\rmdefault}{\mddefault}{\updefault}$\dim 12$}}}
\put(4659,2722){\makebox(0,0)[lb]{{\SetFigFont{5}{6.0}{\rmdefault}{\mddefault}{\updefault}$\dim 12$}}}
\put(1659,2872){\makebox(0,0)[lb]{{\SetFigFont{5}{6.0}{\rmdefault}{\mddefault}{\updefault}representation $V_{4,4}$}}}
\put(1509,22){\makebox(0,0)[lb]{{\SetFigFont{5}{6.0}{\rmdefault}{\mddefault}{\updefault}representation $V_{2,2,2,2}$}}}
\put(4659,2272){\makebox(0,0)[lb]{{\SetFigFont{8}{9.6}{\rmdefault}{\mddefault}{\updefault}cubic $\cubic$}}}
\put(4659,847){\makebox(0,0)[lb]{{\SetFigFont{8}{9.6}{\rmdefault}{\mddefault}{\updefault}quintic $\quintic$}}}
\put(5634,2272){\makebox(0,0)[lb]{{\SetFigFont{8}{9.6}{\rmdefault}{\mddefault}{\updefault}$\proj^{13}$}}}
\put(5409,847){\makebox(0,0)[lb]{{\SetFigFont{8}{9.6}{\rmdefault}{\mddefault}{\updefault}$\subset$}}}
\put(204.000,472.000){\arc{390.000}{0.3948}{5.8884}}
\blacken\path(282.729,618.022)(384.000,547.000)(328.066,657.324)(282.729,618.022)
\blacken\path(328.066,286.676)(384.000,397.000)(282.729,325.978)(328.066,286.676)
\end{picture}
}
\end{center}
\caption{Diagram of interrelationships.}
\label{f:figureeight}
\end{figure}
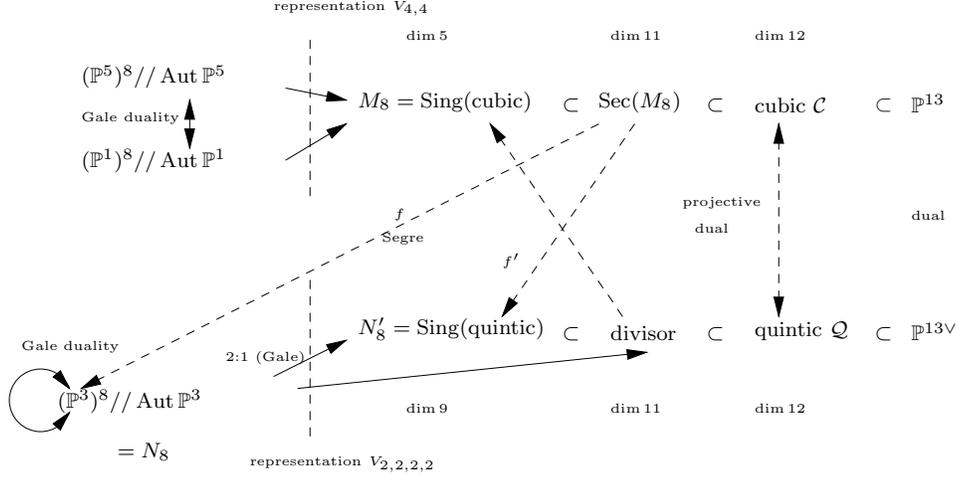

The ring $R_{\bullet}(M_8) = \bigoplus_k \Gamma(\mathcal{O}_{\proj^1}(k)^{\boxtimes 8})^{\SL(2)}$ is generated in
degree $1$ (Kempe's 1894 theorem, see for example \cite[Thm.~2.3]{hmsv1}), and $\dim R_1(M_8) = 14$ (\S \ref{ss:p1}).
We thus have a natural closed immersion  $M_8 \hookrightarrow \proj^{13}$.
By Schur--Weyl duality or a comparison of
tableaux descriptions  (\S
\ref{ss:p1}--\ref{ss:p2}), $R_1(M_8)$ carries the irreducible $\SS_8$-representation $V_{4,4}$.

The ideal of relations of $M_8$,
\begin{displaymath}
I_{\bullet}(M_8) := \ker(\Sym^{\bullet}  R_1(M_8) \to R_{\bullet}(M_8)),
\end{displaymath}
is generated by 14 quadratic relations (Corollary~\ref{generation}, known earlier by computer calculation as described in
(B) above).  There is (up to multiplication by non-zero scalar) a
unique skew-invariant cubic (an element of
$\Sym^3 R_1(M_8)$, Proposition~\ref{p:repfacts}(a)).  We call this cubic the {\em skew cubic} $\cubic$, and by abuse of
notation we call the corresponding hypersurface $\cubic$ as well.  
 The fact that $M_8$ lies on the skew cubic $\cubic$ is a consequence
  of Thomae's formula for hyperelliptic curves.  (We thank Sam
  Grushevsky explaining this to us.)
But more is true ---  the fivefold $M_8$ is the singular locus of
$\cubic$ in a strong sense: $I_{\bullet}(M_8)$ is the Jacobian ideal of $\cubic$ --- the 14 partial derivatives of
the skew cubic $\cubic$ generate $I_{\bullet}(M_8)$  and are of course the 14 quadratic relations described above
(\S \ref{s:threeone}).  (In fact, this result holds away from characteristic $3$.  In characteristic $3$, the Euler
formula yields a linear syzygy among the 14 quadratic relations, and the skew cubic $\cubic$ can be taken as the
remaining generator of the ideal.)

The ring $R_{\bullet}(N_8) = \bigoplus_k \Gamma(\mc{O}_{\proj^3}(k)^{\boxtimes 8})^{\SL(4)}$ is generated in degree
$1$ and $2$ (Proposition~\ref{browncow}), and $\dim R_1(N_8)=14$.  As an $\SS_8$-module, $R_1(N_8)$ is irreducible of type
$V_{2,2,2,2}$ (as with $R_1(M_8)$, by Schur--Weyl duality, \S \ref{ss:p2}, or by direct comparison of the tableaux
description).  The Gale-invariant subalgebra is the subalgebra $R_{\bullet}(N'_8) \subset
R_{\bullet}(N_8)$ generated by $R_1(N_8)$.  More precisely:  we {\em define} the graded
ring $R_{\bullet}(N'_8)$ as the subalgebra of $R_{\bullet}(N_8)$ generated in degree $1$ (i.e., by $R_1(N_8)$),
and {\em define} $N'_8 = \Proj R_{\bullet}(N'_8)$, then {\em show} 
 (in Proposition~\ref{anotherbrowncow})
that $R_{\bullet}(N'_8)$ is the Gale-invariant
subalgebra.

Bezout's theorem implies $\Sec(M_8) \subset \cubic$:  restricting the cubic form $\cubic$ to any line joining two
distinct points of $M_8$ yields a cubic vanishing to order $2$ at those two points (as $M_8 = \Sing \cubic$), so this
cubic must vanish on the line.  The secant variety $\Sec(M_8)$ has
dimension $11$  as one would expect (Corollary~\ref{c:2142}(a)), and is
thus a divisor on the $12$-fold $\cubic$.

Let
\begin{displaymath}
I_{\bullet}(N'_8) := \ker \left( \Sym^{\bullet}(R_1(N'_8)) \to R_{\bullet}(N'_8) \right)
\end{displaymath}
be the ideal of relations of $N'_8$.  By comparing the readily computable $\SS_8$-representations $\Sym^5(R_1(N'_8))$
and $R_5(N'_8)$, we find that there is a skew quintic {\em relation} $\quintic$ in $I_5(N'_8)$
(Proposition~\ref{p:repfacts}(c) and  Theorem~\ref{t:Qdual}; uniqueness
is shown later in Proposition~\ref{uniqueskewquintic}).  Furthermore, the ninefold $N'_8$ is the singular locus of
$\quintic$ (Theorem~\ref{t:singQ}).

\begin{figure}[ht]
\begin{center}
\setlength{\unitlength}{0.00083333in}
\begingroup\makeatletter\ifx\SetFigFont\undefined%
\gdef\SetFigFont#1#2#3#4#5{%
  \reset@font\fontsize{#1}{#2pt}%
  \fontfamily{#3}\fontseries{#4}\fontshape{#5}%
  \selectfont}%
\fi\endgroup%
{\renewcommand{\dashlinestretch}{30}
\begin{picture}(3488,1597)(0,-10)
\put(1350,622){\makebox(0,0)[lb]{{\SetFigFont{5}{6.0}{\rmdefault}{\mddefault}{\updefault}duality contracts}}}
\path(825,1147)(600,472)
\path(975,997)(825,547)
\path(1050,1447)(975,997)
\path(965.136,1120.299)(975.000,997.000)(1024.320,1110.435)
\path(825,172)(825,547)
\path(855.000,427.000)(825.000,547.000)(795.000,427.000)
\path(2025,1297)(2100,922)
\path(2047.049,1033.786)(2100.000,922.000)(2105.883,1045.553)
\path(825,547)(2100,922)
\blacken\path(1993.341,859.359)(2100.000,922.000)(1976.411,916.921)(1993.341,859.359)
\path(975,997)(2100,922)
\blacken\path(1978.270,900.049)(2100.000,922.000)(1982.261,959.916)(1978.270,900.049)
\path(900,772)(2100,922)
\blacken\path(1984.648,877.347)(2100.000,922.000)(1977.206,936.884)(1984.648,877.347)
\path(300,1297)(302,1296)(307,1295)
	(316,1294)(328,1291)(343,1287)
	(362,1283)(382,1278)(404,1273)
	(427,1267)(452,1261)(478,1255)
	(505,1248)(535,1240)(567,1231)
	(600,1222)(636,1212)(666,1203)
	(689,1197)(705,1193)(715,1191)
	(721,1191)(725,1191)(729,1191)
	(734,1190)(743,1187)(756,1182)
	(775,1174)(798,1162)(825,1147)
	(850,1131)(871,1117)(887,1106)
	(897,1100)(902,1097)(905,1096)
	(906,1097)(907,1097)(910,1095)
	(915,1089)(923,1078)(936,1059)
	(954,1032)(975,997)(990,970)
	(1004,943)(1018,917)(1030,894)
	(1042,874)(1052,856)(1062,841)
	(1071,827)(1079,815)(1088,803)
	(1095,792)(1103,780)(1110,766)
	(1116,751)(1122,734)(1127,715)
	(1131,693)(1132,670)(1130,646)
	(1125,622)(1113,597)(1097,577)
	(1080,564)(1063,556)(1047,553)
	(1033,553)(1019,556)(1006,559)
	(993,564)(978,567)(961,570)
	(940,570)(917,568)(889,563)
	(858,556)(825,547)(789,537)
	(759,527)(736,520)(720,514)
	(710,509)(704,506)(700,503)
	(696,501)(691,498)(682,494)
	(669,489)(650,484)(627,478)
	(600,472)(571,467)(549,463)
	(536,461)(530,460)(529,460)
	(531,459)(533,460)(532,460)
	(525,461)(510,463)(484,467)
	(450,472)(421,477)(392,483)
	(364,489)(337,495)(311,502)
	(286,508)(262,515)(239,521)
	(217,527)(197,533)(180,538)
	(167,542)(158,545)(153,546)(150,547)
\path(2400,1222)(2434,1231)(2465,1238)
	(2494,1242)(2518,1244)(2537,1245)
	(2553,1245)(2565,1245)(2574,1244)
	(2581,1242)(2588,1241)(2594,1239)
	(2601,1236)(2610,1232)(2622,1227)
	(2638,1220)(2657,1211)(2681,1200)
	(2710,1186)(2741,1168)(2775,1147)
	(2805,1126)(2834,1104)(2860,1083)
	(2883,1063)(2904,1046)(2922,1032)
	(2938,1019)(2952,1009)(2965,1000)
	(2977,992)(2988,984)(2998,977)
	(3008,968)(3018,958)(3028,946)
	(3039,931)(3048,913)(3058,892)
	(3066,867)(3072,838)(3076,806)
	(3075,772)(3070,741)(3061,711)
	(3051,683)(3039,658)(3027,635)
	(3015,615)(3003,598)(2992,582)
	(2981,569)(2970,556)(2960,545)
	(2950,534)(2939,524)(2928,514)
	(2915,503)(2901,492)(2885,481)
	(2867,468)(2846,455)(2823,442)
	(2796,429)(2766,416)(2734,405)
	(2700,397)(2666,393)(2633,392)
	(2602,394)(2574,397)(2550,403)
	(2529,409)(2512,415)(2497,423)
	(2485,430)(2474,437)(2465,445)
	(2456,453)(2448,462)(2439,470)
	(2429,480)(2417,490)(2404,502)
	(2389,514)(2371,528)(2350,544)
	(2327,561)(2302,580)(2276,600)
	(2250,622)(2224,646)(2201,669)
	(2182,689)(2165,707)(2152,722)
	(2141,735)(2133,744)(2126,751)
	(2121,757)(2116,762)(2112,766)
	(2109,770)(2106,775)(2104,781)
	(2101,788)(2098,798)(2096,811)
	(2094,827)(2092,847)(2093,870)
	(2095,895)(2100,922)(2108,949)
	(2119,975)(2131,998)(2143,1019)
	(2154,1037)(2164,1053)(2174,1066)
	(2183,1077)(2191,1087)(2199,1095)
	(2206,1103)(2214,1111)(2223,1119)
	(2233,1127)(2245,1136)(2258,1147)
	(2275,1158)(2294,1170)(2316,1184)
	(2342,1197)(2370,1210)(2400,1222)
\put(900,1522){\makebox(0,0)[lb]{{\SetFigFont{5}{6.0}{\rmdefault}{\mddefault}{\updefault}$8$ points in $\mathbb{P}^1$}}}
\put(375,22){\makebox(0,0)[lb]{{\SetFigFont{5}{6.0}{\rmdefault}{\mddefault}{\updefault}$8$ points in $\mathbb{P}^1$}}}
\put(1800,1372){\makebox(0,0)[lb]{{\SetFigFont{5}{6.0}{\rmdefault}{\mddefault}{\updefault}$8$ points in $\mathbb{P}^1 \times \mathbb{P}^1 \subset \mathbb{P}^3$}}}
\put(0,922){\makebox(0,0)[lb]{{\SetFigFont{8}{9.6}{\rmdefault}{\mddefault}{\updefault}$\operatorname{Sec}(M_8)$}}}
\put(1050,322){\makebox(0,0)[lb]{{\SetFigFont{8}{9.6}{\rmdefault}{\mddefault}{\updefault}$M_8$}}}
\put(3150,472){\makebox(0,0)[lb]{{\SetFigFont{8}{9.6}{\rmdefault}{\mddefault}{\updefault}$N_8$}}}
\put(1575,547){\makebox(0,0)[lb]{{\SetFigFont{5}{6.0}{\rmdefault}{\mddefault}{\updefault}secant lines}}}
\path(600,1222)(450,472)
\end{picture}
}
\end{center}
\caption{The contraction of the secant variety of $M_8= (\proj^1)^8 \cq
  \SL(2) $ to $N_8 = (\proj^3)^8 \cq \SL(4)$.}
\label{f:paris}
\end{figure}
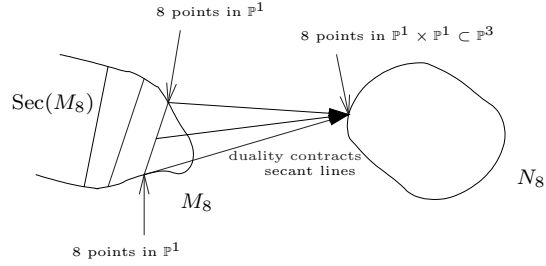

Moreover, $\cubic$ and $\quintic$ are dual hypersurfaces in the sense of projective geometry (Theorem~\ref{t:Qdual}).
\begin{displaymath}
\xymatrix{
\cubic \ar@/_/@{-->}[r]_D &
\quintic \ar@/_/@{-->}[l]_{D'} }
\end{displaymath}
Every secant line $\ell = \overline{pq}$ to $M_8$ (where $p,q \in
M_8$, $p \neq q$) is contracted by the dual map $D: \cubic
\dashrightarrow \quintic$: the dual map is given by the 14 partial derivatives of $\cubic$; their restrictions to
$\ell$ are 14 quadratic relations vanishing at the same two points $p, q$, so they are the same up to scalar.  Thus
$\Sec(M_8)$ is contained in the exceptional divisor of the dual map $D: \cubic \dashrightarrow \quintic$, and in fact
is the entire exceptional divisor (\S \ref{s:singQ}).  Thus the dual map $D$ contracts $\Sec(M_8)$ to $\Sing(\quintic)
= N'_8$.  Furthermore, this map $\Sec(M_8) \dashrightarrow N'_8$ lifts to $\Sec(M_8) \dashrightarrow N_8$, and this
map can be interpreted geometrically as follows (Theorem~\ref{t:cooltheorem}, see Figure~\ref{f:paris}).  Suppose we are
given a point of $\Sec(M_8)$ on a line connecting two general points of $M_8$.  This corresponds to two ordered
octuples of points on $\proj^1$, or equivalently an ordered octuple of points on $\proj^1 \times \proj^1$.  Embedding
$\proj^1 \times \proj^1$ by the Segre map yields $8$ points in $\proj^3$, and hence a point of $N_8$.  The rational map
$\Sec(M_8) \dashrightarrow N_8$ must contract $2$ dimensions ($\dim \Sec(M_8) = 11$ while $\dim N'_8 = 9$); one is
the contraction of the secant line, and the other corresponds to the fact that there is a pencil of quadrics passing
through $8$ points in $\proj^3$.

Although it is not clear from the above description, Theorem~\ref{t:cooltheorem} is the hook on which the rest of
the argument hangs.

We conjecture that the interrelationships of Figure~\ref{f:figureeight} can be completed as follows.

\begin{conj}
The skew quintic $\quintic$ is the trisecant variety (the union of trisecant {\em lines}) of $N'_8$.  The divisor
contracted to $M_8$ by the dual map $D': \quintic \dashrightarrow \cubic$ is the quadrisecant variety (union of
$4$-secant lines) of $N'_8$.\label{conj}
\end{conj}

As evidence, note that the trisecant variety to $N'_8$ lies in the skew quintic, by Bezout's theorem, even though a
naive dimension count suggests that the trisecants should ``easily cover'' all of $\proj^{13}$.  Similarly, Bezout's
theorem implies that the quadrisecant variety to $N'_8$ lies in the contracted divisor (analogous to the above argument
showing that secant lines to $M_8$ are contracted by the dual map),
even though  a naive dimension count suggests that the
quadrisecants should ``easily cover'' all of $\proj^{13}$.

\subsection{Other manifestations of this space, and this graded ring}
\label{othermanifestations}

The extrinsic and intrinsic geometry of $M_n := (\proj^1)^n \!\cq_{1^n} \SL(2)$ for small $n$ has special meaning often
related to the representation theory of $\SS_n$.  For example, $M_4$ relates to the cross ratio, $M_5$ is the quintic
del Pezzo surface, and the Segre cubic $M_6$ has well known remarkable
geometry (see \cite{hmsv6} for further discussion).
The space $M_8$ may be the last of the $M_n$ with such individual personality.  For example, over $\C$, the space may
be interpreted as a ball quotient in two ways:
\begin{enumerate}
\item Deligne and Mostow \cite{dm} showed that $M_8$ is isomorphic to the Satake-Baily-Borel compactification of an
arithmetic quotient of the $5$-dimensional complex ball, using the theory of periods of a family of curves that are
fourfold cyclic covers of $\P^1$ branched at the $8$ points.
\item Kondo \cite{kondo} showed that $M_8$ may also be interpreted in terms of moduli of certain K3 surfaces, and thus
$M_8$ is isomorphic to the Satake-Baily-Borel compactification of a quotient of the complex $5$-ball by $\Gamma(1-i)$,
an arithmetic subgroup of a unitary group of a hermitian form of signature $(1,5)$ defined over the Gaussian integers.
See also \cite[p.~12]{fs} for details  and discussion.
\end{enumerate}
Both interpretations are $\SS_8$-equivariant (see \cite[p.~8]{kondo} for the second).

Similarly, the graded ring $R_{\bullet}(M_8)$ has a number of manifestations:
\begin{enumerate}
\item It is isomorphic to the full ring of modular forms of $\Gamma(1-i)$ \cite[p.~2]{fs}, via the Borcherds additive
lifting.
\item It is the space of sections of multiples of a certain line bundle on $\ol{\mc{M}}_{0,8}$ (as there is a morphism
$\ol{\mc{M}}_{0,8} \to M_8$, \cite{kapranov}, see also \cite{avritzerlange}).
\item Igusa \cite{igusa} showed that there is a natural
  (non-surjective)  map $A(\Gamma_3[2])/\mc{I}_3[2]^0 \to R_{\bullet}(M_8)$, where
$A(\Gamma_3[2])$ is the ring of Siegel modular forms of weight 2 and genus 3. (See \cite[\S 3]{fs} for more discussion.)\item It is a quotient of the third in a sequence of algebras related to the orthogonal group $\mathrm{O}(2m, \F_2)$
defined by Freitag and Salvati Manni, see \cite{fs1}, \cite[\S 2]{fs}.  (The cases $m=5$ and $m=6$ are related to
Enriques surfaces.)
\end{enumerate}

One reason for $M_8$ to be special is the coincidence $\SS_8 \cong \mathrm{O}(6, \F_2)$.  A geometric description of
this isomorphism in this context is given in \cite[\S 4]{fs}.  Another reason is Deligne and Mostow's table
\cite[p.~86]{dm}.

\subsection{Miscellaneous facts about $M_8$ and $N_8$}
\label{s:miscalgebra}

We collect miscellaneous facts about $M_8$ and $N_8$ in case they
prove useful.  The graded free resolution is given in
Proposition~\ref{prop:betti}.\cut{[put it here; gradedness appears to
  be new; F-SM don't do this --- Ben said Jul 13 that F-SM undoubtedly
  had this, but didn't publish it].}  The Hilbert function $f(k) =
\dim{R_k(M_8)}$ follows from this, but was computed classically (see
for example \cite[p.~155, \S 5.4.2.3]{howe}):
\begin{displaymath}
f(k) = \tfrac{1}{3} \left(  k^5 + 5  k^4 + 11 k^3+ 13 k^2+9k+3 \right)
\end{displaymath}
(note this is the same as the Hilbert polynomial),
from which the Hilbert series $\sum_{k=0}^\infty f(k) t^k$ is 
\begin{equation}
\label{howenowe}
\frac{1+8t+22t^2+8t^3+t^4}{(1-t)^6}.
\end{equation}
(Both formulas are given in \cite[p.~7]{fs}.)  The degree of
$M_8$ is 40 (the sum of the coefficients of the numerator, or by the
method of \cite[p. 190]{hmsv1}).
Of course $M_8$ is projectively normal, by what is sometimes called the first fundamental
theorem of invariant theory.  It is arithmetically Gorenstein, as the
numerator
of the Hilbert series is symmetric
\cite[Corollary~4.4.6]{BrunsHerzog}.
Thus the $a$-invariant is
$-2$  (see Proposition~\ref{prop:goren}).
   It doesn't satisfy the $N_2$ condition
of Green and Lazarsfeld:  from the minimal graded free
resolution of \S \ref{s:freeresolution}  the $14$ quadric relations have
nonlinear syzygies.
It  is not Koszul (as the dual Hilbert series $1/H(-t)$ has negative
coefficients, and for Koszul algebras this cannot happen, see for
example \cite[equ.\ (1)]{p}).

By computer calculation, one may show that \cut{Ben  July 12, 2010 e}the  Hilbert series for $N_8$
is
\begin{displaymath}
\frac{1 + 4 t + 31 t^2 + 40 t^3 + 31 t^4 + 4 t^5 + t^6}{(1-t)^{10}},
\end{displaymath}
from which we see that $N_8$ is arithmetically Gorenstein, and the
$a$-invariant is $-4$.    
Another way to see that $N_8$ is Gorenstein is to apply a result of
F. Knop \cite{knop} that given a linear action of a group on affine
space that  preserves volume (i.e.\ it is a subgroup of
$\SL$),  such that  the unstable locus has
codimension at least $2$,   the subring of invariants is
Gorenstein.\cut{Ben Jul 13.  I've got text up front.  (Knop also
shows that the $a$-invariant for $n$ equally weighted points in
$\proj^m$ is $-\gcd(n,m+1)$, which explains the $a$-invariants of $-2$
and $-4$ above.)}
One may similarly compute that the Hilbert series for $N'_8$ is 
$$
\frac { 1 + 4 t + 10 t^2 + 20 t^3 + 21 t^4} { (1-t)^{10}}
$$ from which   $\deg N'_8 = 56$.

\subsection{Relation to the six-point case}
\label{s:relationtosix}

(We will not need this picture, so we omit all details.)  The classical geometry of six points in projective space,
Figure~\ref{f:figsix}, shows strong similarities to Figure~\ref{f:figureeight}.  This can be made more precise in a
number of ways.  Here is one way to see Figure~\ref{f:figsix} ``at the boundary'' of Figure~\ref{f:figureeight}.  In
the space of 8 points in $\proj^3$ (the bottom left of Figure~\ref{f:figureeight}), consider the locus where the two
given points (of the eight) coincide.  Projecting from that point of $\proj^3$, the remaining six points (generally)
give six points in $\proj^2$ (the bottom left of Figure~\ref{f:figsix}).  This can be extended to all parts of the two
Figures, in a way respecting the Gale and projective dualities.   

\begin{figure}[ht]
\begin{center}
\setlength{\unitlength}{0.00083333in}
\begingroup\makeatletter\ifx\SetFigFont\undefined%
\gdef\SetFigFont#1#2#3#4#5{%
  \reset@font\fontsize{#1}{#2pt}%
  \fontfamily{#3}\fontseries{#4}\fontshape{#5}%
  \selectfont}%
\fi\endgroup%
{\renewcommand{\dashlinestretch}{30}
\begin{picture}(5247,2458)(0,-10)
\put(5109,697){\makebox(0,0)[lb]{{\SetFigFont{8}{9.6}{\rmdefault}{\mddefault}{\updefault}$\proj^{4 \vee}$}}}
\path(2034,1822)(3684,1672)
\blacken\path(3561.777,1652.987)(3684.000,1672.000)(3567.209,1712.741)(3561.777,1652.987)
\path(2034,1222)(3684,1522)
\blacken\path(3571.302,1471.018)(3684.000,1522.000)(3560.569,1530.050)(3571.302,1471.018)
\path(2034,1147)(3684,772)
\blacken\path(3560.335,769.341)(3684.000,772.000)(3573.633,827.849)(3560.335,769.341)
\path(2034,322)(5109,622)
\blacken\path(4992.480,580.490)(5109.000,622.000)(4986.654,640.206)(4992.480,580.490)
\blacken\path(1239.000,1552.000)(1209.000,1672.000)(1179.000,1552.000)(1239.000,1552.000)
\path(1209,1672)(1209,1372)
\blacken\path(1179.000,1492.000)(1209.000,1372.000)(1239.000,1492.000)(1179.000,1492.000)
\blacken\path(4089.000,1402.000)(4059.000,1522.000)(4029.000,1402.000)(4089.000,1402.000)
\dashline{60.000}(4059,1522)(4059,847)
\blacken\path(4029.000,967.000)(4059.000,847.000)(4089.000,967.000)(4029.000,967.000)
\dashline{60.000}(2334,2047)(2334,172)
\path(3684,622)(2034,397)
\blacken\path(2148.846,442.938)(2034.000,397.000)(2156.953,383.489)(2148.846,442.938)
\put(534,1747){\makebox(0,0)[lb]{{\SetFigFont{8}{9.6}{\rmdefault}{\mddefault}{\updefault}$(\proj^3)^6 // \Aut \proj^3$}}}
\put(534,1222){\makebox(0,0)[lb]{{\SetFigFont{8}{9.6}{\rmdefault}{\mddefault}{\updefault}$(\proj^1)^6// \Aut \proj^1$}}}
\put(3759,1597){\makebox(0,0)[lb]{{\SetFigFont{8}{9.6}{\rmdefault}{\mddefault}{\updefault}Segre cubic}}}
\put(3759,697){\makebox(0,0)[lb]{{\SetFigFont{8}{9.6}{\rmdefault}{\mddefault}{\updefault}Igusa quartic}}}
\put(534,322){\makebox(0,0)[lb]{{\SetFigFont{8}{9.6}{\rmdefault}{\mddefault}{\updefault}$(\proj^2)^6 // \Aut \proj^2$}}}
\put(2859,997){\makebox(0,0)[lb]{{\SetFigFont{5}{6.0}{\rmdefault}{\mddefault}{\updefault}Veronese}}}
\put(2559,622){\makebox(0,0)[lb]{{\SetFigFont{5}{6.0}{\rmdefault}{\mddefault}{\updefault}Gale-fixed}}}
\put(1734,2347){\makebox(0,0)[lb]{{\SetFigFont{8}{9.6}{\rmdefault}{\mddefault}{\updefault}outer automorphism}}}
\put(159,772){\makebox(0,0)[lb]{{\SetFigFont{5}{6.0}{\rmdefault}{\mddefault}{\updefault}Gale duality}}}
\put(1359,1522){\makebox(0,0)[lb]{{\SetFigFont{5}{6.0}{\rmdefault}{\mddefault}{\updefault}Gale duality}}}
\put(1809,2122){\makebox(0,0)[lb]{{\SetFigFont{5}{6.0}{\rmdefault}{\mddefault}{\updefault}representation $V_{3,3}$}}}
\put(1809,22){\makebox(0,0)[lb]{{\SetFigFont{5}{6.0}{\rmdefault}{\mddefault}{\updefault}representation $V_{2,2,2}$}}}
\put(4134,1222){\makebox(0,0)[lb]{{\SetFigFont{5}{6.0}{\rmdefault}{\mddefault}{\updefault}projective dual}}}
\put(5109,1222){\makebox(0,0)[lb]{{\SetFigFont{5}{6.0}{\rmdefault}{\mddefault}{\updefault}dual}}}
\put(4734,1597){\makebox(0,0)[lb]{{\SetFigFont{8}{9.6}{\rmdefault}{\mddefault}{\updefault}$\subset$}}}
\put(4734,697){\makebox(0,0)[lb]{{\SetFigFont{8}{9.6}{\rmdefault}{\mddefault}{\updefault}$\subset$}}}
\put(5109,1597){\makebox(0,0)[lb]{{\SetFigFont{8}{9.6}{\rmdefault}{\mddefault}{\updefault}$\proj^4$}}}
\put(204.000,397.000){\arc{390.000}{0.3948}{5.8884}}
\blacken\path(282.729,543.022)(384.000,472.000)(328.066,582.324)(282.729,543.022)
\blacken\path(328.066,211.676)(384.000,322.000)(282.729,250.978)(328.066,211.676)
\end{picture}
}
\end{center}
\caption{The classical geometry of six points in projective space
  (cf.\ Figure~\ref{f:figureeight}).}
\label{f:figsix}
\end{figure}
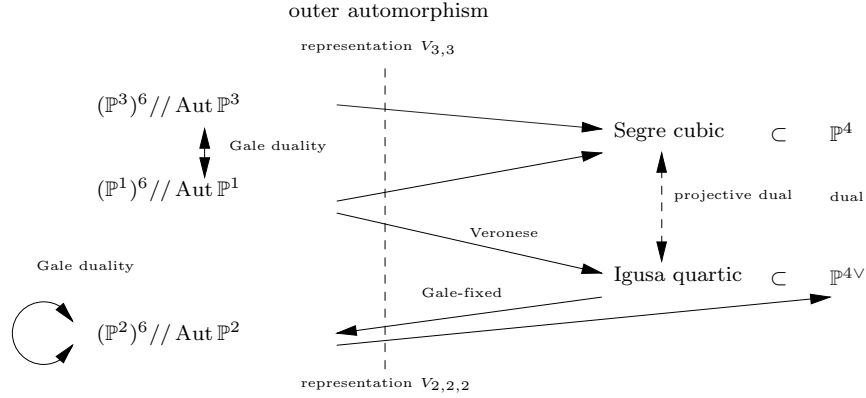

\subsection{Acknowledgments}  

Foremost we thank Igor Dolgachev, who predicted the existence of the cubic of $\cubic$ to us.  Diane Maclagan and
Greg Smith gave essential advice on computational issues at key points
in this project.  We also thank Daniel Erman, Sam Grushevsky,
Shrawan Kumar, Riccardo Salvati Manni, and Larry O'Neil for helpful comments.

\section{Preliminaries on invariant theory and representation theory}

\subsection{Invariants of $n$ points in $\proj^{m-1}$ (with linearization $1, \dots, 1$)}
\label{ss:p1}

(See \cite{dolgachev} for a thorough introduction to all invariant theory facts we need.) The degree $d$ invariants of
$n$ points in $\proj^{m-1}$ are generated (as a vector space over a ground field $\fieldk$, or more generally as a
module over a ground ring) by invariants corresponding to certain tableaux:  $m \times (dn/m)$ matrices, with entries
consisting of the numbers $1$ through $n$, each appearing $d$ times.  To such a tableau, we associate a product of
$m \times m$ determinants, one for each column.  To each column, we associate the $m \times m$ determinant whose $i$th
row consists of the projective coordinates of the point indexed by the entry in that row.  For example, if $m=d=2$ and
$n=4$, and the four points in $\proj^1$ have coordinates $[x_i: y_i]$ ($1 \leq i \leq 4$), then corresponding to 
\begin{displaymath}
\begin{array}{|c|c|c|c|} \hline
1 & 2 & 1 & 4 \\  \hline
3 & 3 & 4 & 2 \\  \hline
\end{array}
\end{displaymath}
we associate the $\SL(2)$-invariant 
\begin{displaymath}
(x_1 y_3 - x_3 y_1) ( x_2 y_3 - x_3 y_2) (x_1 y_4 - x_4 y_1) (x_4 y_2- x_2 y_4).
\end{displaymath}
The linear relations among these invariants are spanned by three basic types:  (i) columns can be rearranged without
changing the invariant (obvious); (ii) swapping two entries in the same column changes the sign of the invariant
(obvious); and (iii) Pl\"ucker or straightening relations, which we do not describe here (see \cite[\S 1.3]{hmsv2} for
a graphical description).  The straightening algorithm implies that for fixed $n$, $m$, $d$, the semistable tableaux
(where the entries are increasing vertically and weakly increasing horizontally) form a basis.

If $m=2$ and $n$ is even, it is not hard to see (and a theorem of Kempe, see for example \cite[Thm.~2.3]{hmsv1}) that
the ring of invariants is generated in degree $1$.  Thus the GIT quotient $(\proj^1)^n \cq \SL(2)$ naturally comes with
a projective embedding, whose coordinates correspond to $2 \times
(n/2)$ tableaux.
It is helpful  to interpret the invariants as directed graphs on  $n$ vertices, where for each column
$\begin{array}{|c|} \hline
i\\  \hline
j\\  \hline
\end{array}$
we draw an edge $\vec{ij}$ (see \cite[\S 1.2]{hmsv2}).  In this language, there is a basis consisting of
upwards-oriented non-crossing graphs (those graphs with only edges $\vec{ij}$ with $j>i$, where when represented
with the vertices cyclically arranged around a circle, no two edges cross).  This basis is different than the one
provided by semi-standard tableaux.  As an example, Figure~\ref{fig:noncross} gives a basis for $R_1(M_8)$.
The following information is omitted to highlight the symmetries:  the vertices are
labeled cyclically $1$ through $8$ (it does not matter to us where one starts), and edges are upwards-oriented (if $i<j$,
edge $ij$ is oriented $\vec{ij}$).

\begin{figure}[ht]
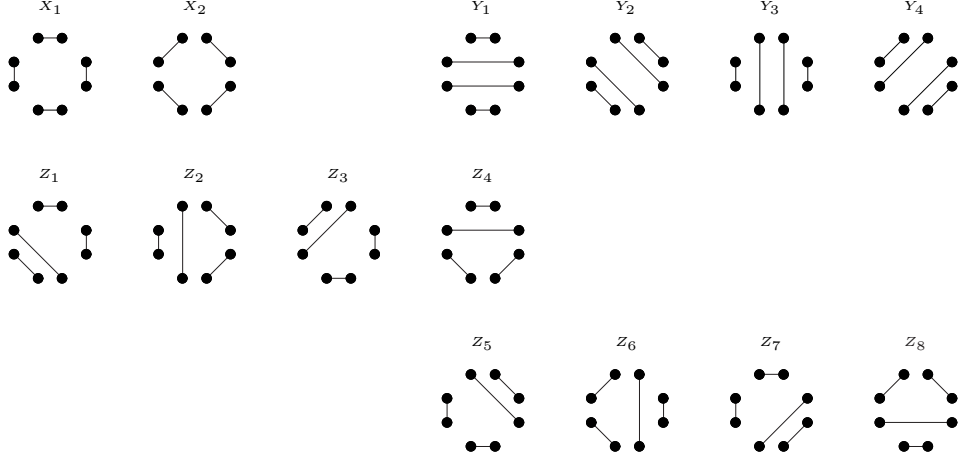

\begin{center}
\include{fourteen}
\end{center}
\caption{The fourteen non-crossing matchings on eight points.}
\label{fig:noncross}
\end{figure}

If $m$ is arbitrary, $d=1$, and $n$ is divisible by $m$, the description of the degree $1$ invariants, with its $\SS_n$
action, is precisely the usual tableaux description of the irreducible $\SS_n$-representation $V_{(n/m)^m}$.  If $n=8$
and $m=2$ or $m=4$, the corresponding representation has dimension $14$ (see Fig.~\ref{fig:noncross} for the former),
so $\dim R_1(M_8) = \dim R_1(N_8) = 14$.

If $n=8$ and $m=2$, we have the quadratic relation
\begin{displaymath}
\begin{array}{|c|c|c|c|}  \hline
1 & 3 & 5 & 7 \\ \hline 
2 & 4 & 6 & 8 \\ \hline 
\end{array} \times 
\begin{array}{|c|c|c|c|}  \hline
1 & 2 & 5 & 6 \\ \hline 
3 & 4 & 7 & 8 \\ \hline 
\end{array} = 
\begin{array}{|c|c|c|c|}  \hline
1 & 3 & 5 & 6 \\ \hline 
2 & 4 & 7 & 8 \\ \hline 
\end{array} \times
\begin{array}{|c|c|c|c|}  \hline
1 & 2 & 5 & 7 \\ \hline 
3 & 4 & 6 & 8 \\ \hline 
\end{array} \, .
\end{displaymath}
(By ``relation'' we mean the difference between the two sides is an element of $\Sym^2(R_1(M_8))$ mapping to 0 in
$R_2(M_8)$.)
This is clearly a relation:  each column appears the same number of times on each side.  All four tableaux are
semistandard, so this equation is non-zero.  This is an example of a \emph{simple binomial relation}, central to
\cite{hmsv2}.  With appropriate labelling of vertices, in terms of the variables of Figure~\ref{fig:noncross}, the
relation is
\begin{equation}
\label{e:simplebinomial}
X_2 Y_1 = Z_4 Z_8.
\end{equation}

\subsection{Representation-theoretic preliminaries: $\SS_8$-decomposition of ideals}
\label{ss:p2}

Recall that we are working over a field $\fieldk$ of characteristic
$0$ (although all statements hold over $\Z[1/8!]$).
We will repeatedly use Schur--Weyl duality:  for a vector space $W$ and a positive integer $n$, we have a canonical
decomposition
\begin{displaymath}
W^{\otimes n}=\bigoplus_{\lambda} S_{\lambda}(W) \otimes V_{\lambda},
\end{displaymath}
where the sum is over partitions $\lambda$ of $n$ and $S_{\lambda}$ denotes the Schur functor associated to $\lambda$.
This decomposition is compatible with the commuting actions of $\SS_n$ and $\GL(W)$ on each side.  If $\lambda$ has
more parts than the dimension of $W$ then $S_{\lambda}(W)=0$, so one can restrict the sum to those partitions having at
most $\dim{W}$ parts.  For such partitions, the spaces $S_{\lambda}(W)$ form mutually non-isomorphic irreducible
representations of $\GL(W)$. As an example, let $n=8$ (which will be the case throughout this paper) and let $W=\C^2$.
Let $\lambda=(a,b)$ be a partition of 8 into two parts ($b \le a$ by convention).  The $\GL(W)$-representation
$S_{\lambda}(W)$ is isomorphic to $(\det W)^b \otimes (\Sym^{a-b} W)$.  This has an $\SL(W)$ invariant if and only if
$a=b$, i.e., if $a=b=4$.  We thus see that $R_1(M_8) = (W^{\otimes 8})^{\SL(W)}$ is isomorphic to the
$\SS_8$-representation $V_{4,4}$.

The decomposition of $I_d(M_8)$ into irreducible $\SS_8$-representations may be determined as follows:
\begin{displaymath}
I_d(M_8) = \ker( \Sym^d (R_1(M_8)) \twoheadrightarrow R_d(M_8)),
\end{displaymath}
and $\Sym^d(R_1(M_8))$ may be determined from character theory (using the fact that $R_1(M_8)$ carries the
representation $V_{4,4}$), and the representation on
\begin{displaymath}
R_d(M_8) = \Gamma( (\proj^1)^8, \oh_{(\proj^1)^8}(d, \dots, d))^{\SL(2)}
\end{displaymath}
can be determined by Schur-Weyl duality. 

Similarly, information about the decomposition of $I_d(N'_8)$ into irreducible $\SS_8$-representations can be readily
determined by the map
\begin{displaymath}
I_d(N'_8) = \ker( \Sym^d (R_1(N_8)) \to R_d(N_8)).
\end{displaymath}
Caution: the map $\Sym^d (R_1(N_8)) \to R_d(N_8)$ is \emph{not} in general a surjection --- the analogue of Kempe's
theorem does not hold (see  Propositions~\ref{browncow} and~\ref{anotherbrowncow}).

The particular facts we  need are the following.  The first was proved with $8$ replaced by arbitrary even $n$ in
\cite[Prop.~6.5]{hmsv2}, but can be verified for $n=8$ as described above, or using the methods of
Proposition~\ref{p:repfacts}(a) below.

\begin{proposition}
\label{prop:decomp}
We work over a characteristic $0$ field $\fieldk$.  In the following table, each representation is multiplicity free.
The set of irreducible representations it contains corresponds to the given set of partitions.
\begin{center}
\rm
\begin{tabular}{c|c}
$\SS_8$-representation & Set of partitions of 8 \\[.5ex]
\hline \\[-2ex]
$\Sym^2(R_1(M_8))$ & at most four parts, all even \\[1ex]
$\bw{R_1(M_8)}$ & exactly four parts, all odd \\[1ex]
$R_1(M_8)^{\otimes 2}$ & union of previous two sets \\[1ex]
$R_2(M_8)$ & at most three parts, all even \\[1ex]
$I_2(M_8)$ & exactly four parts, all even
\end{tabular}\end{center}
\end{proposition}

As described in the introduction, it has been checked by brute force computer calculation (by Maclagan, Koike, and
Freitag and Salvati Manni) that the quadratic relations generate the ideal of relations, and a pure thought proof is
given here (see Corollary~\ref{generation}).

\begin{proposition}
\label{p:repfacts}
We work over a characteristic $0$ field $\fieldk$.  All statements refer to $\SS_8$-representations.
\begin{enumerate}
\item[(a)] {\em ``The skew cubic relation for $M_8$.''}
Up to scalar, there is a single skew-invariant in $\Sym^3 R_1(M_8)$, and it is a relation, i.e., it lies in
\begin{displaymath}
I_3(M_8) = \ker \left( \Sym^3 R_1(M_8) \twoheadrightarrow R_3(M_8) \right).
\end{displaymath}
\item[(b)] {\em ``The fourteen quadratic relations for $M_8$.''}
The degree $2$ part of the ideal of $M_8$ is a single representation of type $V_{2,2,2,2}$:
\begin{displaymath}
I_2(M_8) = \ker \left( \Sym^2 R_1(M_8) \twoheadrightarrow R_2(M_8) \right) \cong V_{2,2,2,2}.
\end{displaymath}
\item [(c)] {\em ``The skew quintic relation for $N'_8$.''}
There is a non-zero skew-invariant  relation in $\Sym^5 R_1(N_8)$ vanishing on $N_8$, i.e.,
\begin{displaymath}
I_5(N'_8) = \ker \left( \Sym^5 R_1(N'_8) \rightarrow R_5(N'_8) \right)
\end{displaymath}
contains a skew-quintic.
\item[(d)] {\em ``The fourteen quartic relations for $N'_8$.''}
There is a representation of type $V_{4,4}$ in the degree $4$ part of the ideal of $N'_8$, i.e., in 
\begin{displaymath}
I_4(N'_8) = \ker \left( \Sym^4 R_1(N'_8) \rightarrow \Sym^4 (N'_8) \right).
\end{displaymath}
\end{enumerate}
\end{proposition}

We will verify uniqueness in (c) in
Proposition~\ref{uniqueskewquintic}:  there is a one non-zero skew quintic
relation up to scalar.
We will verify uniqueness in (d) in Corollary~\ref{cor:I4}:  there
is precisely one representation of type $V_{4,4}$ in $I_4(N'_8)$.  

\begin{proof}
To prove (a), first verify that $\Sym^3 R_1(M_8)$ has a single $\sgn$ component by character theory.\cut{As alternative:  we need to show that there is at least one; then to show that there is only one.  To show
      there is at least one, refer to the argument given in the original paper. To show there is only one, use the
      fact that this is the cubic hypersurface.  Or the only one with these 14 derivatives.  I want to think
      this through.}
Then note that $R_3(M_8)$ has no $\sgn$ component: $(\Sym^3( \fieldk^2))^{\otimes 8}$ has no $\sgn$ component because
by Schur-Weyl duality it contains no $\SS_8$-representation with more than $4 = \dim (\Sym^3(\fieldk^2))$ rows.
(Alternatively, as in \cite[\S 2]{hmsvCR}, use the fact that the
Vandermonde has too high degree.
As another alternative,
 an explicit formula for this skew invariant is given in
Remark~\ref{r:explicit}.)

Part (b) follows from Proposition~\ref{prop:decomp}.  (Alternatively, use the method of part (a).)

Parts (c) and (d) follow from comparing the appropriate representations in $\Sym^d( R_1(N'_8))$ and $R_d(N'_8)$ for
$d=4,5$, using Schur-Weyl duality for $R_d(N_8)$.  (The sign representation $\sgn$ appears with multiplicity $4$ in
$\Sym^5(R_1(N'_8))$, but multiplicity $3$ in $R_5(N_8)$.  The representation $V_{4,4}$ appears with multiplicity $7$
in $\Sym^4(R_1(N'_8))$, but multiplicity $6$ in $R_4(N_8)$.)  Note that we only get bounds on the multiplicities
since $\Sym^{\bullet}(R_1(N'_8))$ does not surject onto $R_{\bullet}(N_8)$.
\end{proof}

\section{The web of relationships between $M_8$ and $N_8$, via the skew cubic $\cubic$  and the skew quintic $\quintic$}
\label{s:dualities}

\subsection{The skew cubic relation $\cubic$}
\label{s:threeone}

Let $\cubic$ be the skew-invariant cubic of Proposition~\ref{p:repfacts}(a) (which is unique up to scalar).  We also
denote the corresponding hypersurface in $\proj(R_1(M_8)^*)$ by $\cubic$.  An element $\lambda$ of $R_1(M_8)^*$
induces a derivation on the ring $\Sym(R_1(M_8))$, which we denote by $\partial/\partial \lambda$ (we think of it
as taking a partial derivative).  We have a map
\begin{displaymath}
R_1(M_8)^* \otimes \sgn \to \Sym^2(R_1(M_8)), \qquad \lambda \mapsto \frac{\partial \cubic}{\partial \lambda}
\end{displaymath}
which is $\SS_8$-equivariant.  The image is an irreducible representation of type $V_{4,4} \otimes \sgn
=V_{2,2,2,2}$, and is therefore equal to $I_2(M_8)$ by Proposition~\ref{prop:decomp}; in other words, the above map
furnishes a natural isomorphism
\begin{displaymath}
R_1(M_8)^* \otimes \sgn \to I_2(M_8).
\end{displaymath}
``The'' 14 quadrics are the image of the basis of $R_1(M_8)^*$ dual to that of $R_1(M_8)$ given by the 14 planar
graphs.
The above discussion shows that the partial derivatives of $\cubic$ all vanish on $M_8$.  Furthermore, the simple
binomial relations necessarily span the same irreducible representation --- by Proposition~\ref{prop:decomp}, the
quadratic relations form an irreducible $\SS_8$-representation.

Thus the fivefold $M_8$ is contained in the singular locus of $\cubic$.  (As described in \S \ref{mainconstructions},
$M_8$ {\em is} the singular locus of $\cubic$:  we establish this in Corollary~\ref{generation}, though it also
follows by the computer calculations of Maclagan, Koike, and
Freitag and Salvati Manni, or by those of  \cite[Prop.~2.10]{hmsv1}.  We will not need this fact in this section.)

\begin{remark}
\label{r:explicit}
One can describe the cubic explicitly, in terms of the variables of Figure~\ref{fig:noncross}:
\begin{eqnarray*}
\cubic & = & X_1X_2(X_1+X_2) +
X_1X_2(Z_1+Z_2+Z_3+Z_4+Z_5+Z_6+Z_7+Z_8) \\ & & -
(X_1Y_2Y_4 + X_2Y_3Y_1) 
+(X_1Z_2Z_6 + X_2Z_3Z_7 + X_1Z_4Z_8 + X_2Z_5Z_1)
\\ & & +(Y_1Z_2Z_6+Y_2Z_3Z_7+Y_3Z_4Z_8+Y_4Z_5Z_1) 
-(Z_1Z_2Z_3 + Z_2Z_3Z_4 \\ & & + Z_3Z_4Z_5 + Z_4Z_5Z_6 + Z_5Z_6Z_7 +Z_6Z_7Z_8+
Z_7Z_8Z_1+Z_8Z_1Z_2).
\end{eqnarray*}
(One can verify directly that $\SS_8$ acts on the expression above via the $\sgn$ representation as follows.
Cyclically rotating the labels on the eight vertices of the graphs of Figure~\ref{fig:noncross} clearly
changes the sign of $\cubic$.  One readily checks by hand that swapping two chosen adjacent labels changes
the sign of $\cubic$, using the Pl\"ucker relations once.) The connection to the simple binomials is quite
explicit.  For example, $\frac {\partial \cubic } { \partial Y_3} = - X_2 Y_1 + Z_4 Z_8$ is the  simple
binomial relation \eqref{e:simplebinomial}. 
\end{remark}

\begin{remark}
One can also describe the cubic conceptually:  it it the sum of the cubes of the 105 matchings on 8 points, each
weighted by a sign in a systematic manner.  Equivalently, it is the skew-average of the cube of any matching.  These
constructions clearly give skew-invariant cubics,  but  it is
non-trivial to show that they are non-zero.  Details are given in \cite[Prop.~3.1]{hmsve}.\label{march10}
\end{remark}

\begin{proposition}
\label{p:gregsmith}
Suppose the ground field $\fieldk$ is $\Q$.  Let $\hessian_{\cubic}$ be the Hessian  of $\cubic$ (the determinant of
the $14 \times 14$ Hessian matrix, degree $14$).  Then (the scheme corresponding to)
$\hessian_{\cubic}$ does not contain $\cubic$ (so $\deg (\hessian_{\cubic} \cap \cubic) = 42$), and the irreducible
components of $\hessian_{\cubic} \cap \cubic$ have degree $21$ or $42$.
\end{proposition}

This proof uses  the only computer calculation we need in \S \ref{s:dualities}.  The calculation makes
essential use of the fact that $\fieldk$ is $\Q$.

We will later (Corollary~\ref{c:21} and Proposition~\ref{p:H2}) deduce that $\hessian_{\cubic}$ meets $\cubic$ along an irreducible subvariety
($\Sec (M_8)$) of degree $21$, with multiplicity $2$, and that this holds over {\em any} field $\fieldk$ of
characteristic $0$.

\begin{proof}
We choose a suitable plane $\proj^2 \subset \proj^{13}$ over $\Q$, and observe by computer that the intersection of
$\hessian_{\cubic} \cap \proj^2$ with $\cubic \cap \proj^2$ is an irreducible degree $21$ (dimension $0$) subscheme,
appearing  with multiplicity $2$.  (Short Macaulay2 code is given at \cite{hmsvcode}.)  The result follows.
\end{proof}

\subsection{The (projective) dual map from $\cubic$ contracts $\Sec(M_8)$}

The cubic $\cubic$ is naturally a subscheme of $\proj(R_1(M_8)^*)$.
The dual map $D: \cubic \dashrightarrow \proj(R_1(M_8))$ (sending a smooth point of $\cubic$ to its tangent space)
is the rational map corresponding to the map on rings $\Sym(R_1(M_8)^*) \to \Sym(R_1(M_8))/(\cubic)$ which maps
$\lambda \in R_1(M_8)^*$ to $\frac{\partial \cubic}{\partial \lambda}$.  The two representations $R_1(M_8)$ and
$N_1(M_8)^*$ of $\SS_8$ differ by the sign character.  We therefore have a \emph{canonical} isomorphism
$\proj(R_1(M_8))=\proj(R_1(N_8)^*)$.  We regard $D$ as mapping to $\proj(R_1(N_8)^*)$.
Note that $D$ blows up the singular locus of $\cubic$, which includes $M_8$.

The dual to the cubic $\cubic$ is a hypersurface:  $\hessian_{\cubic} \cap \cubic \neq \cubic$ by
Proposition~\ref{p:gregsmith} (this can also be checked  easily by hand), so $\cubic$ is not contracted by the dual
map.

As argued in \S \ref{mainconstructions}, Bezout's theorem implies that $\Sec(M_8) \subset \cubic$, and every secant
line to $M_8$ is contracted by the dual map, so $\Sec(M_8)$ is contained in the exceptional divisor of the dual map
$D: \cubic \dashrightarrow \proj(R_1(N_8)^*)$.  Note that the construction of $N_8$ gives a map 
$N_8 \to \proj(R_1(N_8)^*)$.  The image is $N'_8$ (by the definition of $N'_8$).

\begin{theorem}
\label{t:cooltheorem}
Under the duality map $D$, the space $\Sec(M_8)$ maps dominantly to $N'_8$.
\end{theorem}

Before proving the theorem, we introduce an auxiliary map and establish a few of its properties.  The Segre map gives
an embedding $(\P^1)^8 \times (\P^1)^8 \to (\P^3)^8$, which descends to a rational map $\sigma:M_8 \times M_8 \dashrightarrow N_8$.
(We only get a rational map since a pair of stable points in $(\P^1)^8$ need not map to a stable point of
$(\P^3)^8$.)  The following two lemmas give the properties of this map that we need.

\begin{lemma}
The map $\sigma$ is dominant.  
\end{lemma}

\begin{proof}
We may assume $\fieldk$ is algebraically closed.
Let $x$ be a general point of $N_8$, which we regard as 8 general points $x_1, \ldots, x_8$ in $\P^3$.  Through these 8
points passes a one parameter family of quadrics (since the space of quadrics in $\P^3$ is 9 dimensional), a
generic member $Q$ of which is smooth.  The group $\SL(4)$ acts
transitively on the smooth quadrics in $\P^3$ --- this is equivalent to the fact that any two non-degenerate quadratic
forms on $\fieldk^4$ are equivalent.  Thus after moving $x_1, \ldots, x_8$ by an element of $\SL(4)$ (which does not
affect $x$), we can assume that these 8 points lie on the image of the Segre map $\P^1 \times \P^1 \to \P^3$.  Thus
each $x_i$ gives rise to a point $(y_i, y_i')$ on $\P^1 \times \P^1$, and so we get two points $y=(y_1,\ldots,y_8)$
and $y'=(y_1',\ldots,y'_8)$ on $(\P^1)^8$.  It is clear that $\sigma(y,y')=x$, which proves the lemma.
\end{proof}

For two graded rings $A_{\bullet}$ and $B_{\bullet}$, we write
$A_{\bullet} \boxtimes B_{\bullet}$ for the projective
coordinate ring of $\Proj(A_{\bullet}) \times_{\fieldk}
\Proj(B_{\bullet})$ ---  the graded ring
whose degree $n$ piece is $A_n \otimes_{\fieldk} B_n$. 

\begin{lemma}
The map $\sigma$ is induced from a map $\sigma^*:R_{\bullet}(N_8) \to R_{\bullet}(M_8) \boxtimes R_{\bullet}(M_8)$
of graded rings.
\end{lemma}

\begin{proof}
The Segre embedding $(\P^1)^8 \times (\P^1)^8 \to (\P^3)^8$ lifts to a map
\begin{displaymath}
(\fieldk^2)^8 \times (\fieldk^2)^8 \to (\fieldk^4)^8.
\end{displaymath}
The ring $R_{\bullet}(M_8) \boxtimes R_{\bullet}(M_8)$ consists of functions on $(\fieldk^2)^8 \times (\fieldk^2)^8$
which are $\SL(2) \times \SL(2)$ invariant and $(\fieldk^{\times})^8 \times (\fieldk^{\times})^8$ semi-invariant.  The ring
$R_{\bullet}(N_8)$ consists of functions on $(\fieldk^4)^8$ which are $\SL(4)$ invariant and $(\fieldk^{\times})^8$
semi-invariant.  It is clear that functions of the one kind pullback to those of the other under the above map.
This pullback map on functions is $\sigma^*$.
\end{proof}

Note that $R_{\bullet}(M_8) \boxtimes R_{\bullet}(M_8)$ is a subring of $R_{\bullet}(M_8) \otimes R_{\bullet}(M_8)$.
(There is a slight change in grading, e.g., $R_1(M_8) \otimes R_1(M_8)$ is degree 1 in the former and degree 2 in
the latter.)  In what follows, we regard $\sigma^*$ as mapping to the latter ring.

We now prove the theorem.

\begin{proof}[Proof of Theorem~\ref{t:cooltheorem}]
Fix an isomorphism $I_2(M_8) \cong R_1(N_8)$ of $\SS_8$-modules.  Consider the diagram of spaces:
\begin{displaymath}
\xymatrix{
\Cone(M_8) \times \Cone(M_8) \ar[r] \ar[d] & \Cone(N_8) \ar[d] \\
R_1(M_8)^* \ar[r] & I_2(M_8)^* \cong R_1(N_8)^* }
\end{displaymath}
Here $\Cone$ denotes the affine cone of a projective variety.  All maps are morphisms of affine schemes, not just
rational maps.  We now explain the maps.  We have a natural inclusion $\Cone(M_8) \subset R_1(M_8)^*$.  The left map
adds its two components inside the vector space $R_1(M_8)^*$.  The right map is the natural map $\Cone(N_8) \to
R_1(N_8)^*$ obtained by interpreting  elements of $R_1(N_8)$ as functions on $\Cone(N_8)$.  The bottom map is the cone on
the duality map $D$; more precisely, it is given by the partial derivatives of $\cubic$.  The top map is
$\Spec(\sigma^*)$ (which is not quite the cone on $\sigma$, but close).

Under the left map, $\Cone(M_8) \times \Cone(M_8)$ maps dominantly to the cone on $\Sec(M_8)$.  The top map
is dominant --- this follows easily from the dominance of $\sigma$.  Under the right map $\Cone(N_8)$ maps
surjectively to $\Cone(N_8')$.  It follows that to prove the theorem it is enough to show that the above diagram
commutes up to multiplication by a non-zero scalar.

Consider the diagram on rings corresponding to the above diagram of spaces:
\begin{displaymath}
\xymatrix{
R_{\bullet}(M_8) \otimes R_{\bullet}(M_8) & R_{\bullet}(N_8) \ar[l] \\
\Sym^{\bullet}(R_1(M_8)) \ar[u] & \Sym^{\bullet}(I_2(M_8)) \cong \Sym^{\bullet}(R_1(N_8)) \ar[l] \ar[u] }
\end{displaymath}
All maps respect the action of $\SS_8$ and respect the various gradings (if defined correctly: one must regrade the
$N_8$ spaces by a factor of 2).
To show that the diagram commutes, it suffices to show that it commutes when restricted to the degree
two elements in the bottom right, since they generate those rings.  The degree two pieces of the bottom right rings
are irreducible representations of $\SS_8$ of type $V_{2,2,2,2}$.  The degree two piece of the top left ring is
$R_1(M_8)^{\otimes 2} \oplus R_2(M_8)^{\oplus 2}$, which by Proposition~\ref{prop:decomp} contains exactly one copy of
$V_{2,2,2,2}$.  It follows that one of the two maps from the bottom right to the top left is a scalar multiple of the
other.  Since each map is non-zero, the scalar is non-zero.
\end{proof}

\begin{remark}
The above proof may seem surprising, since we showed that a diagram was commutative without using very much
about the maps involved.  For instance, we only used three properties of the map $\Cone(M_8) \times \Cone(M_8) \to
R_1(M_8)^*$, namely:  (1) that it is $\SS_8$-equivariant; (2) that the induced map on rings preserves the grading; and (3)
that $I_2(M_8)$ is not contained in the kernel of the map on rings (this was used to conclude that the ``left then
up'' map in the diagram of rings was non-zero).  However, these are very strong conditions to place on a map:  there
is a 2 parameter family of such maps, and they all induce the same rational map $M_8 \times M_8 \to \P(R_1(M_8)^*)$.
So we really did use everything about the map!
\end{remark}

\begin{remark}
Theorem~\ref{t:cooltheorem}
 can also be proved by an easy algebraic computation, as follows.  Let $p$ and $q$ be two generic
points of $M_8$.  We write $p=(p_1, \ldots, p_8)$ and for convenience work in inhomogeneous coordinates, so that
$p_i$ is interpreted a $[1 ; p_i]$ in $\P^1$ (and similarly for $q$).  Let $t$ be a generic number and consider the point $p+tq$ on
$\Sec(M_8)$.  Evaluating $p+tq$ on \eqref{e:simplebinomial} (a partial derivative of the cubic, and so one
coordinate of the dual map) gives
\begin{equation}
\begin{split}
\label{eqn:leftside}
& 
\left(\;
\begin{array}{|c|c|c|c|}  \hline
1 & 3 & 5 & 7 \\ \hline 
2 & 4 & 6 & 8 \\ \hline 
\end{array} \, (1; p_i)_{1 \leq i \leq 8}  \right. 
 +  t
\left. \begin{array}{|c|c|c|c|}  \hline
1 & 3 & 5 & 7 \\ \hline 
2 & 4 & 6 & 8 \\ \hline 
\end{array} \, (1; q_i)_{1 \leq i \leq 8}
\;\right) \\
+& 
\left(\;
\begin{array}{|c|c|c|c|}  \hline
1 & 2 & 5 & 6 \\ \hline 
3 & 4 & 7 & 8 \\ \hline 
\end{array} \, (1; p_i)_{1 \leq i \leq 8}  \right. 
+  t
\left. \begin{array}{|c|c|c|c|}  \hline
1 & 2 & 5 & 6 \\ \hline 
3 & 4 & 7 & 8 \\ \hline 
\end{array} \, (1; q_i)_{1 \leq i \leq 8}
\;\right)\\
-& \left(\;
\begin{array}{|c|c|c|c|}  \hline
1 & 2 & 5 & 7 \\ \hline 
3 & 4 & 6 & 8 \\ \hline 
\end{array} \, (1; p_i)_{1 \leq i \leq 8}  \right. 
+  t
\left. \begin{array}{|c|c|c|c|}  \hline
1 & 2 & 5 & 7 \\ \hline 
3 & 4 & 6 & 8 \\ \hline 
\end{array} \, (1; q_i)_{1 \leq i \leq 8}
\;\right) \\
+& 
\left(\;
\begin{array}{|c|c|c|c|}  \hline
1 & 3 & 5 & 6 \\ \hline 
2 & 4 & 7 & 8 \\ \hline 
\end{array} \, (1; p_i)_{1 \leq i \leq 8} \right. 
+ t
\left. \begin{array}{|c|c|c|c|}  \hline
1 & 3 & 5 & 6 \\ \hline 
2 & 4 & 7 & 8 \\ \hline 
\end{array} \, (1; q_i)_{1 \leq i \leq 8}
\;\right)
\end{split}
\end{equation}
We wish to show that this agrees with the image of $(p,q)$ in $N_8'$ under the Segre map followed by $N_8 \to N'_8$.
A coordinate of the image of $(p,q)$ is given by the expression
\begin{equation}
\label{eq:july16}
t \;  \begin{array}{|c|c|}  \hline
1 & 5 \\ \hline
2 & 6 \\ \hline
3 & 7 \\ \hline 
4 & 8    \\ \hline
\end{array}
\, (1 ; p_i ; q_i ; p_i q_i)_{1 \leq i \leq 8}.
\end{equation}
A short argument or computation shows that the two expressions \eqref{eqn:leftside} and \eqref{eq:july16} ---
both polynomials in the $p_i$, $q_i$ and $t$ --- are equal.  The equality of the other components of the two maps
follows from either similar computations, or an appeal to the $\SS_8$-symmetry.
\end{remark}

\begin{corollary}
We have (a) $\dim \Sec(M_8) = 11$; and (b) $\deg \Sec(M_8) = 21$ or $42$.\label{c:2142}
\end{corollary}

\begin{proof}
(a) Of course  $\dim \Sec (M_8) \leq 2 \dim (M_8)+1 = 11$.  For the opposite inequality, note that $\dim N_8 = 9$, and
we see that the preimage (under the dominant rational map $\Sec(M_8) \dashrightarrow N_8$) of a general point of $N_8$
has dimension at least $2$:  one corresponding to the one-parameter family of quadrics through $8$ general points in
$\proj^3$, and one corresponding to the secant line joining those two points of $M_8$ (corresponding to the two
octuples of points in $\proj^1$).  Thus $\dim \Sec(M_8) \geq 11$. 

Part (b) then follows from Proposition~\ref{p:gregsmith}.  Note that Proposition~\ref{p:gregsmith} assumes the base
field $\fieldk$ is $\Q$, but it suffices to show
Corollary~\ref{c:2142} in this case, as degree is preserved by  extension of base field.
\end{proof}

We pause to take stock of where we are.  We now know that $\hessian_{\cubic} \cap \cubic$, which has degree $42$,
contains $\Sec(M_8)$ (which has degree $21$ or $42$) as a component.  We will soon (Proposition~\ref{p:H2}) see that
$\hessian_{\cubic} \cap \cubic$ contains $\Sec(M_8)$ with multiplicity $2$ (and hence  that $\deg \Sec(M_8) = 21$).

\subsection{There is a skew quintic relation $\quintic$ in
  $I_{\bullet}(N'_8)$ defining the dual
hypersurface to $\cubic$}

As previously mentioned, we have a canonical isomorphism
$\proj(R_1(M_8))=\proj(R_1(N_8)^*)$, and so we can regard the dual
hypersurface $\quintic$ (the reduced image of the dual map) to
$\cubic$ as a subvariety of $\proj(R_1(N_8)^*)$.  As $\quintic$ is
reduced, it is the zero locus of a unique (up to scaling) square-free
polynomial, which we also denote by $\quintic$.  We begin our analysis
of the dual hypersurface with the following result:

\begin{proposition}
\label{p:datleast5}
The degree of $\quintic$ is at least $5$.
\end{proposition}

\cut{
\begin{proof}
The map $D$ contracts $\Sec(M_8)$, and hence sends $\Sec(M_8)$ into the singular
locus of the dual hypersurface $\quintic'$.  The partial derivatives $\frac {\partial \quintic'} {\partial y_i}$ of the
dual hypersurface have degree $d-1$, and vanish on $\Sing \quintic'$ but not on all of  $\quintic'$.  Thus
$D^* \frac {\partial \quintic'} {\partial y_i}$ vanishes on $\Sec (M_8)$ but not $\cubic$.  Now
$D^*\frac{\partial \quintic'} {\partial y_i}$ has degree $2(d-1)$.  But $\deg (\Sec(M_8) ) \geq 21$ (Corollary~\ref{c:2142}),
so by Bezout's theorem, $(\deg D^* \frac {\partial \quintic'} {\partial y_i}) \cap \cubic = 2(d-1) \times 3  \geq 21$,
from which $d \geq 5$.
\end{proof}
}

\begin{proof}
Let $d$ be the degree of $\quintic$.  The map $D$ contracts $\Sec(M_8)$, and hence sends $\Sec(M_8)$ into the singular
locus of $\quintic$.  Let $f$ be a partial derivative of $\quintic$.  Then $f$ has degree $d-1$ and vanishes on 
$\Sing(\quintic)$ but not all of $\quintic$.  Thus $D^*f$ vanishes on $\Sec(M_8)$ but not $\cubic$.  Now, $D^*f$ has
degree $2(d-1)$.  But $\deg(\Sec(M_8)) \geq 21$ (Corollary~\ref{c:2142}), so by Bezout's theorem, $\deg((D^*f) \cap \cubic)
= 2(d-1) \times 3  \geq 21$, from which $d \geq 5$.\end{proof}

We will use the following consequence in \S \ref{s:N8}.

\begin{corollary}
\label{c:united}
The ideal $I_{\bullet}(N'_8)$ contains no relations of degree less than $4$.
\end{corollary}

\begin{proof}
Suppose that $f \in I_d(N'_8)$ is a non-zero relation of degree $d<4$.  Then $D^*f$ does not vanish on $\cubic$ by
Proposition~\ref{p:datleast5}.  Now $D^*f$ is a polynomial of degree $2d$ so by Bezout's theorem,
$\deg(D^* f \cap \cubic) = 6d$.  But $f$ vanishes on $N'_8$, so $D^*f$ vanishes on $f^{-1} N'_8 = \Sec(M_8)$, so
\begin{displaymath}
6d \geq \deg( \Sec(M_8)) \geq 21
\end{displaymath}
(from Corollary~\ref{c:2142}), from which $d \geq 4$, yielding a contradiction.
\end{proof}

\begin{corollary}
We have $\deg \Sec(M_8) = 21$.\label{c:21}
\end{corollary}

\begin{proof}
By Proposition~\ref{p:repfacts} (d), there exists a non-zero  relation
of degree $4$. 
By the same argument as the proof of Corollary~\ref{c:united},
$6\times 4  \geq \deg ( \Sec(M_8))$.    The result then follows from Corollary~\ref{c:2142}.\end{proof}

\begin{proposition}
\label{prop:I4}
There is an isomorphism
\begin{displaymath}
\Phi:I_4(N'_8) \to R_1(M_8),
\end{displaymath}
unique up to scalar, 
characterized by the following property:  if $\mc{R} \in I_4(N'_8)$ is a quartic relation on $N'_8$ then $D^*\mc{R}
\cap \cubic$ is the union (sum of divisors) of $\Sec(M_8)$ and the intersection of $\cubic$ with the hyperplane determined by
$\Phi(\mc{R})$.  
\end{proposition}


\begin{proof}
  First note the hypersurface $\cubic$ is factorial (by a theorem of
  Grothendieck,
  \cite[Exp.\ XI, 3.14]{sga2}, implying that complete intersections
  factorial in codimension $3$ are factorial --- our special case can also be shown by hand
  using Nagata's criterion for factoriality \cite[Lem.\ 19.20]{e}
  applied to the explicit description of the cubic of
  Remark~\ref{r:explicit}).  Thus all Weil divisors are Cartier.
  Also, by the Lefschetz hyperplane theorem for Picard groups,
  $\Pic(\proj(R_1(M_8)^*)) \to \Pic(\cubic)$ is an isomorphism
  \cite[Exp.\ XII, Cor.~3.7]{sga2} (line bundles on complete
  intersections in $\proj^n$ of dimension at most $3$ are all
  restrictions from the ambient projective space).  Now $\Sec (M_8)$ is a
  divisor of degree $21$ on $\cubic$.  Thus $\Sec(M_8)$ is the
  vanishing scheme of some section $s \in \Gamma( \cubic,
  \oh(7)|_{\cubic})$ (unique up to scalar).  We remark that $\SS_8$
thus acts on $s$ by a character, and hence either by the identity by sign.

  We begin by noting that $D^*$ yields a linear map $I_4(N'_8)
  \rightarrow \Gamma(\cubic, \oh(8)|_{\cubic}).$ For any element of
  $I_4(N'_8)$, its pullback by $D$ vanishes on $D^* N'_8 = \Sec(M_8)$,
  and thus is divisible by (the effective Cartier divisor) $s$.
  Dividing by $s$ yields a map $I_4(N'_8) \rightarrow \Gamma(\cubic,
  \oh(1)|_{\cubic})$.  From the long exact sequence associated to $$0
  \rightarrow \oh_{\proj^{14}}(-2) \rightarrow \oh_{\proj^{14}}(1)
  \rightarrow \oh(1)|_{\cubic} \rightarrow 0,$$ and $\Gamma(\proj^{14},
  \oh(1)) = R_1(M_8)$, we may identify $\Gamma(\cubic, \oh(1)|_{\cubic})$
  with $R_1(M_8)$.  We have thus obtained a map $I_4(N'_8) \rightarrow
  R_1(M_8)$, unique up to scalar.    It is not the zero map.
\end{proof}

\begin{corollary}
\label{cor:I4}
The space $I_4(N'_8)$ is isomorphic to $V_{4,4}$ as a representation of $\SS_8$.
\end{corollary}

\begin{proof}
We know $I_4(N'_8)$ contains a copy of $V_{4,4}$, while the proposition shows that $I_4(N'_8)$ is 14 dimensional.
\end{proof}

We remark that as $I_4(N'_8)$ and
  $R_1(M_8)$ are both the representation $V_{4,4}$, 
the section  $s$ appearing in the proof of Proposition~\ref{prop:I4},
 cutting out $\Sec(M_8)$,  must be $\SS_8$-invariant.

\begin{theorem}
\label{t:Qdual}
The polynomial $\quintic$ is degree 5, is skew-invariant under $\SS_8$, and its derivatives belong to $I_4(N'_8)$.
Furthermore, it is the unique such polynomial, up to scalars.
\end{theorem}

\begin{remark}
Since the derivatives of $\quintic$ belong to $I_{\bullet}(N'_8)$, the Euler formula implies that $\quintic$ itself
belongs to $I_{\bullet}(N'_8)$.  (See the proof below.)
We will verify that $\quintic$ is in fact the {\em unique} skew
quintic relation in Proposition~\ref{uniqueskewquintic}. 
One might hope that the skew quintic is the signed sum of fifth powers
of Specht coordinates, in analogy with the situation for the cubic
(see Remark~\ref{march10}).  We can show that this signed sum is non-zero
by an analogous method, it is unfortunately {\em not}  a relation for $N_8$.
\end{remark}

\begin{proof}
The proof of Proposition~\ref{prop:I4} re-interprets the dual map $D': \quintic \dashrightarrow \cubic$ as coming from the linear system of the ``14 quartics relations of $N'_8$.''
We now make this precise.

Choose an isomorphisms of $\SS_8$-modules
$\alpha_1:R_1(M_8) \to R_1(N_8)^* \otimes \sgn$.
Let $\alpha_2:R_1(M_8) \to I_4(N'_8)$ be inverse to the isomorphism $\Phi$ of Proposition~\ref{prop:I4}.
Let
\begin{displaymath}
D'_1, D'_2:\P(R_1(N_8)^*) \dashrightarrow \P(R_1(M_8)^*)
\end{displaymath}
be the maps corresponding to the ring maps $\Sym(R_1(M_8)) \to \Sym(R_1(N_8))$ given by mapping $x \in R_1(M_8)$ to
$\frac{\partial \quintic}{\partial \alpha_1(x)}$ and $\alpha_2(x)$, respectively.

Then $D'_1|_{\quintic}$ is (essentially by definition) the dual map $D'$, and Proposition~\ref{prop:I4} interprets the 14-dimensional family of quartic relations of $N'_8$, after pulling back by $D$ to $\cubic$ and subtracting the base locus $\Sec(M_8)$, with the 14-dimensional family of sections of $\oh_{\proj R_1(M_8)}(1)$, so $D'_2|_{\cubic}$ is also the dual map $D'$.

Now the quartics in $I_4(N'_8)$ have no common divisorial component on the hypersurface $\quintic$. (Otherwise
their quotient by this divisor would be a lower-degree polynomial, which when pulled back to $\cubic$, would vanish
on $\Sec(M_8)$ but not on all of $\cubic$, which is impossible by Bezout's theorem.)  Thus, since $D'_1=D'_2$ on $\quintic$, there
must be a non-zero homogeneous polynomial $P$ such that
\begin{equation}
\label{eq4}
\frac{\partial \quintic}{\partial \alpha_1(x)} = P \alpha_2(x)
\end{equation}
for all $x \in R_1(M_8)$.  (A priori, this equality should be taken modulo $\quintic$, but 
the left side has degree less than $\deg \quintic$.)  By the following lemma, $P$ must be a scalar, and so by scaling $\alpha_1$ we can assume $P=1$.

Since $P$ has degree 0, the formula \eqref{eq4} implies that $\quintic$ has degree 5.  Let $\{x_i\}$ be a basis for
$R_1(N_8)$ and let $\{x_i^*\}$ be the dual basis.  The Euler formula
\begin{displaymath}
5 \quintic=\sum_{i=1}^{14} x_i \frac{\partial \quintic}{\partial x_i^*}=\sum_{i=1}^{14} x_i
\alpha_2(\alpha_1^{-1}(x_i^*))
\end{displaymath}
shows that $\quintic$ belongs to $I_5(N'_8)$ (since the image of $\alpha_2$ is $I_4(N'_8)$).  For $g \in \SS_8$, we
have
\begin{displaymath}
\frac{\partial \quintic}{\partial \alpha_1(gx)}=\alpha_2(gx)=g \alpha_2(x)=\frac{\partial (g \quintic)}{\partial
(g\alpha_1(x))}=\sgn(g) \frac{\partial (g \quintic)}{\partial \alpha_1(gx)}.
\end{displaymath}
The appearance of $\sgn$ in the above equation comes from the twist by $\sgn$ in the definition of $\alpha_1$.  The
above equation shows that all the derivatives of $\sgn(g) \quintic$ and $g \quintic$ agree, which implies that
$g \quintic=\sgn(g) \quintic$, i.e., $\quintic$ is skew-invariant under $\SS_8$.

We now verify that $\quintic$ is the unique skew-invariant element of $I_5(N'_8)$ (up to scalars) whose derivatives
belong to $I_4(N'_8)$.  Thus let $\quintic'$ be an arbitrary such element.  Let $\beta:R_1(N_8)^* \otimes \sgn
\to I_4(N'_8)$ be the map $\lambda \mapsto \frac{\partial \quintic}{\partial \lambda}$ and let $\beta'$ be the
analogous map for $\quintic'$.  Because $\quintic$ is skew-invariant, the map $\beta$ is $\SS_8$-equivariant, and
similarly for $\beta'$.  Since $I_4(N'_8)$ is irreducible under $\SS_8$, we
have $\beta'=c \beta$ for some scalar
$c$.  This implies that $\quintic'=c\quintic$, as was to be shown.
\end{proof}

\begin{lemma}
Let $Q$ be a square-free homogeneous polynomial of degree $>1$ in several variables over a field $\fieldk$.  Then the
partial derivatives of $Q$ have no common factor of degree $\ge 1$.
\end{lemma}

\begin{proof}
Assume for the sake of contradiction that $F$ is an irreducible polynomial of degree $\ge 1$ dividing all the partial
derivatives of $Q$.  Necessarily, $F$ is homogeneous.  We have $dQ=0$ on the irreducible hypersurface $F=0$, and so
$Q$ is constant on $F=0$.  Since $Q$ and $F$ are homogeneous, we actually have $Q=0$ on $F=0$ and so $Q=FG$ for some
homogeneous polynomial $G$.  Let $x$ be an indeterminate appearing in $F$.  Then $\partial_x Q=(\partial_x G)F+
(\partial_x F) G$.  Since $F$ divides $\partial_x Q$ by assumption and $\partial_x F$ is non-zero and coprime to $F$,
we find that $F$ divides $G$.  This shows that $F^2$ divides $Q$, a contradiction.
\end{proof}


\subsection{The singular locus of $\quintic$ is $N'_8$}
\label{s:singQ}

The dual map $D': \quintic \dashrightarrow \cubic$ blows up precisely $\Sing \quintic$, which is cut out by the 14
quartics spanning $I_4(N'_8)$.  The exceptional divisor on $\cubic$ is the intersection of the pull-back of these
quartics under $D$, which (from the proof of Proposition~\ref{prop:I4}) is $\Sec (M_8)$.  But $D(\Sec(M_8)) = N'_8$.  Thus we have shown
the following.

\begin{theorem}
\label{t:singQ}
We have $\Sing(\quintic) = N'_8$.
\end{theorem}

Also, the Hessian $\hessian_{\cubic}$ vanishes precisely along the exceptional divisor, so we can refine Proposition~\ref{p:gregsmith} to the following.

\begin{proposition}
\label{p:H2}
The Hessian $\hessian_{\cubic}$ vanishes to order $2$ along $\Sec(M_8)$.
\end{proposition}

Thus the Hessian is a perfect square modulo $\cubic$. 
Specifically, if $s$ is the invariant section of $\oh(7)|_{\cubic}$ appearing in the proof of Proposition~\ref{prop:I4}, any $\overline{s}$ is any lift of $s$ to a section of $\oh(7)$ 
(which can be taken to be invariant), then the Hessian is $\overline{s}^2$ modulo $\cubic$.

We make a small remark about another invariant septic. The pullback of the skew quintic $D^* \quintic$ to $\proj(R_1(M_8)^*)$ has degree 10, and vanishes on the skew cubic $\cubic$.  The residual divisor to $\cubic$ in $D^* \quintic$ is an invariant septic.  This septic contains $M_8$.  Because we will not use this fact, we omit the proof.


\section{The partial derivatives of $\cubic$ have no linear syzygies I}
\label{s:syzygy1}

The goal of \S \ref{s:syzygy1}--\ref{s:syzygy3} is to establish the following:

\begin{theorem}
\label{prop:nosyz}
The partial derivatives of $\cubic$ have no linear syzygies.
\end{theorem}

We now elaborate on the statement of the theorem.  Pick a basis $x_1, \ldots, x_{14}$ of $R_1(M_8)$ (we can then make
sense of $\frac{\partial \cubic}{\partial x_i}$ by defining it to be $\frac{\partial \cubic}{\partial x_i^*}$, where
$x_i^*$ is the dual basis).  The theorem is then the statement that if $y_1, \ldots, y_{14}$ are elements of
$R_1(M_8)$ such that $\sum_{i=1}^{14} y_i \frac{\partial \cubic}{\partial x_i}=0$ then $y_i=0$ for all $i$.  We will
give a more canonical reformulation of this statement below.  In this section, we will reduce the proof of
Theorem~\ref{prop:nosyz} to a problem that we will solve in \S \ref{s:syzygy3}.  We first note an important consequence.

\begin{corollary}
\label{prop:i2-i3}\label{generation}
The ideal $I_{\bullet}(M_8)$ is generated in degree $2$.
\end{corollary}

Another proof, avoiding the complicated toric degeneration of \cite{hmsv2}, will be given in \S \ref{s:freeresolution},
see Remark~\ref{r:denver}.

\begin{proof}
By \cite[Thm.~5.1]{hmsv2}, $I_{\bullet}(M_8)$ is cut out by quadratics and cubics.  One can readily check by hand
(counting noncrossing graphs) that $\dim (I_3(M_8)) = 14^2$.  By Theorem~\ref{prop:nosyz}, the map of $196$-dimensional
vector spaces $R_1(M_8) \otimes I_2(M_8) \to I_3(M_8)$ has no kernel and is thus surjective.
\end{proof}

\begin{remark}
 As remarked in \S \ref{mainconstructions}, Theorem~\ref{prop:nosyz} holds away from characteristic $3$.  In
characteristic $3$, the Euler formula yields a linear syzygy among the 14 quadratic relations, and the skew cubic
$\cubic$ is the remaining generator of the ideal.  See \cite[Thm 1.2
and  \S 9]{hmsve} for a proof.  
This argument requires a computer, unlike the proof of Theorem~\ref{prop:nosyz}.
\end{remark}

We prove Theorem~\ref{prop:nosyz} by the following strategy.

\begin{enumerate}
\item[(a)] Let $\Psi:\mf{gl}(R_1(M_8)) \to I_3(M_8)$ be the map defined via a natural action of the Lie algebra
$\mf{gl}(R_1(M_8)) \cong \mf{gl}(14)$ on $\Sym^3(R_1(M_8))$ (described below).  We first observe that the space of
linear syzygies between the partial derivatives of $\cubic$ is exactly $\mf{g}=\ker{\Psi}$.  We note that $\mf{g}$ is a
Lie subalgebra of $\mf{sl}(R_1(M_8))$ and is stable under the action of $\SS_8$.
\item[(b)] Next, using general theory developed in \S \ref{s-liesub} concerning $G$-stable Lie subalgebras of
$\mf{sl}(V)$, where $V$ is a representation of $G$, and the classification of simple Lie algebras, we show that the
only $\SS_8$-stable Lie subalgebras of $\mf{sl}(R_1(M_8))$ are 0, $\mf{so}(14)$ and $\mf{sl}(14)$.  Thus $\mf{g}$ must
be one of these three Lie algebras.
\item[(c)] Finally, we show that $\mf{so}(14)$ does not annihilate any non-zero cubic.  As $\mf{g}$ is the annihilator
of $\cubic$ (and $\mf{so}(14) \subset \mf{sl}(14)$), we conclude $\mf{g}=0$.
\end{enumerate}

We now implement this strategy.  Consider the composition
\begin{displaymath}
\begin{split}
\wt{\Psi}:& \End(R_1(M_8)) \otimes \Sym^3(R_1(M_8)) = R_1(M_8) \otimes
R_1(M_8)^* \otimes \Sym^3(R_1(M_8)) \to \\
& R_1(M_8) \otimes \Sym^2(R_1(M_8)) \to \Sym^3(R_1(M_8))
\end{split}
\end{displaymath}
where the first map is the partial derivative map and the second map is the multiplication map.  One easily verifies
that $\wt{\Psi}$ is just the map which expresses the action of the Lie algebra $\mf{gl}(R_1(M_8))=\End(R_1(M_8))$ on the
third symmetric power of its standard representation $R_1(M_8)$.  We are trying to show that $\wt{\Psi}$ induces an
injection
\begin{displaymath}
\Psi:\End(R_1(M_8)) \otimes \fieldk \cubic \to I_3(M_8).
\end{displaymath}
(We know that $\Psi$ maps $\End(R_1(M_8)) \otimes \fieldk \cubic$ into $I_3(M_8)$ since we know that the partial
derivatives of $\cubic$ belong to $I_2(M_8)$.)  Indeed, the kernel of $\Psi$ is the space of linear syzygies between the
partial derivatives of $\cubic$.  Now, the kernel of $\Psi$ is equal to $\mf{g} \otimes \fieldk \cubic$, where $\mf{g}$
is the annihilator in $\mf{gl}(R_1(M_8))$ of $\cubic$.  Thus Theorem~\ref{prop:nosyz} is equivalent to the following:

\begin{proposition}
\label{prop:l-zero}
We have $\mf{g}=0$.
\end{proposition}

We know two important things about $\mf{g}$:  first, $\mf{g}$ is a Lie subalgebra of $\mf{gl}(R_1(M_8))$, as it is the
annihilator of some element in a representation of $\mf{gl}(R_1(M_8))$; and second, $\mf{g}$ is stable under $\SS_8$,
as the action map $\Psi$ is $\SS_8$-equivariant and $\fieldk \cubic$ is stable under $\SS_8$.  We will prove
Proposition~\ref{prop:l-zero} by first classifying the $\SS_8$-stable Lie subalgebras of $\mf{gl}(R_1(M_8))$ and then
proving that $\mf{g}$ cannot be any of them except zero.

Before continuing, we note the following result:

\begin{proposition}
The Lie algebra $\mf{g}$ is contained in $\mf{sl}(R_1(M_8))$.
\end{proposition}

\begin{proof}
The trace map $\mf{gl}(R_1(M_8)) \to \fieldk$ is $\SS_8$-equivariant, where $\SS_8$ acts trivially on the target
$\fieldk$.  Thus if $\mf{g}$ contained an element of non-zero trace it would have to contain a copy of the trivial
representation.  By Proposition~\ref{prop:decomp}, $\mf{gl}(R_1(M_8)) \cong R_1(M_8)^{\otimes 2}$ is multiplicity free
as an $\SS_8$-representation.  Thus the one-dimensional space spanned by the identity matrix is the only copy of the
trivial representation in $\mf{gl}(R_1(M_8))$.  Therefore, if $\mf{g}$ were not contained in $\mf{sl}(R_1(M_8))$ then
it would contain the center of $\mf{gl}(R_1(M_8))$.  However, we know that the identity matrix does not annihilate
$\cubic$.  Thus $\mf{g}$ must be contained in $\mf{sl}(R_1(M_8))$.
\end{proof}

%
%

\section{Interlude:  $G$-stable Lie subalgebras of $\mf{sl}(V)$}
\label{s-liesub}

In this section $G$ will denote an arbitrary finite group and $V$ an irreducible representation of $G$ over an
algebraically closed field $\fieldk$ of characteristic zero.  We investigate the following general problem:

\begin{problem}
Determine the $G$-stable Lie subalgebras of $\mf{sl}(V)$.
\end{problem}

We do not obtain a complete answer to this question, but we prove strong enough results to determine the answer in our
specific situation. We will use the term \emph{$G$-subalgebra} to mean a $G$-stable Lie subalgebra.

\subsection{Some structure theory}

Our first result is the following:

\begin{proposition}
\label{prop:solv}
Let $V$ be an irreducible representation of $G$.  Then every solvable $G$-subalgebra of $\mf{sl}(V)$ is abelian and
consists of semi-simple elements.
\end{proposition}

\begin{proof}
Let $\mf{g}$ be a solvable subalgebra of $\mf{sl}(V)$.  By Lie's theorem, $\mf{g}$ preserves a complete flag $0=V_0
\subset \cdots \subset V_n=V$.  The action of $\mf{g}$ on each one-dimensional space $V_i/V_{i-1}$ must factor through
$\mf{g}/[\mf{g}, \mf{g}]$; thus $[\mf{g}, \mf{g}]$ acts by zero on $V_i/V_{i-1}$ and so carries $V_i$ into $V_{i-1}$.
The space $[\mf{g}, \mf{g}] V$ is therefore not all of $V$.  On the other hand, $[\mf{g}, \mf{g}]$ is $G$-stable and
therefore so is $[\mf{g}, \mf{g}] V$.  From the irreducibility of $V$ we conclude $[\mf{g}, \mf{g}] V=0$, from which it
follows that $[\mf{g}, \mf{g}]=0$.  Thus $\mf{g}$ is abelian.

Now let $R$ be the subalgebra of $\End(V)$ generated (under the usual multiplication) by $\mf{g}$.  Let $R_s$
(resp.\ $R_n$) denote the set of semi-simple (resp.\ nilpotent) elements of $R$.  Then $R_s$ is a subring of $R$, $R_n$
is an ideal of $R$ and $R=R_s \oplus R_n$.  As $R_n^m=0$ for some $m$, the space $R_n V$ is not all of $V$.  As it is
$G$-stable it must be zero, and so $R_n=0$.  We thus find that $R=R_s$ and so all elements of $R$, and thus all elements
of $\mf{g}$, are semi-simple.
\end{proof}

Let $V$ be a representation of $G$.  We say that $V$ is \emph{imprimitive} if there is a decomposition
$V=\bigoplus_{i \in I} V_i$ of $V$ into non-zero subspaces, at least two in number, such that each element of $G$
carries each $V_i$ into some $V_j$.  We say that $V$ is \emph{primitive} if it is not imprimitive.  Note that primitive
implies irreducible.  An irreducible representation is imprimitive if and only if it is induced from a proper subgroup.

\begin{proposition}
\label{prop:prim}
Let $V$ be an irreducible representation of $G$.  Then $V$ is primitive if and only if the only abelian $G$-subalgebra
of $\mf{sl}(V)$ is zero.
\end{proposition}

\begin{proof}
Let $V$ be an irreducible representation of $G$ and let $\mf{g}$ be a non-zero abelian $G$-subalgebra of $\mf{sl}(V)$.
We will show that $V$ is imprimitive.  By Proposition~\ref{prop:solv} all elements of $\mf{g}$ are semi-simple.  We
thus get a decomposition $V=\bigoplus V_{\lambda}$ of $V$ into eigenspaces of $\mf{g}$ (each $\lambda$ is a linear map
$\mf{g} \to \fieldk$).  As $\mf{g}$ is $G$-stable, each element of $G$ must carry each $V_{\lambda}$ into some
$V_{\lambda'}$.  Note that if $V=V_{\lambda}$ for some $\lambda$ then $\mf{g}$ would consist of scalar matrices, which
is impossible as $\mf{g}$ is contained in $\mf{sl}(V)$.  Thus there must be at least two non-zero $V_{\lambda}$ and so
$V$ is imprimitive.

We now establish the other direction.  Thus let $V$ be an irreducible imprimitive representation of $G$.  We construct
a non-zero abelian $G$-subalgebra of $\mf{sl}(V)$.  Write $V=\bigoplus V_i$ where the elements of $G$ permute the $V_i$.
Let $p_i$ be the endomorphism of $V$ given by projecting onto $V_i$ and then including back into $V$ and let $\mf{g}$ be
the subspace of $\mf{gl}(V)$ spanned by the $p_i$.  Then $\mf{g}$ is an abelian subalgebra of $\mf{gl}(V)$ since
$p_i p_j=0$ for $i \ne j$.  Furthermore, $\mf{g}$ is $G$-stable since for each $i$ we have $g p_i g^{-1}=p_j$ for some
$j$.  Intersecting $\mf{g}$ with $\mf{sl}(V)$ gives a non-zero abelian $G$-subalgebra of $\mf{sl}(V)$ (the intersection
is non-zero because $\mf{g}$ has dimension at least two and $\mf{sl}(V)$ has codimension one).
\end{proof}

We have the following important consequence of Proposition~\ref{prop:prim}:

\begin{corollary}
\label{cor:prim-ss}
Let $V$ be a primitive representation of $G$.  Then every $G$-subalgebra of $\mf{sl}(V)$ is semi-simple.
\end{corollary}

\begin{proof}
Let $\mf{g}$ be a $G$-subalgebra of $\mf{sl}(V)$.  The radical of $\mf{g}$ is then a solvable $G$-subalgebra and
therefore vanishes.  Thus $\mf{g}$ is semi-simple.
\end{proof}

Proposition~\ref{prop:prim} can also be used to give a criterion for primitivity.

\begin{corollary}
\label{cor:prim-test}
Let $V$ be an irreducible representation of $G$ such that each non-zero $G$-submodule of $\mf{sl}(V)$ has dimension at
least that of $V$.  Then $V$ is primitive.
\end{corollary}

\begin{proof}
Let $\mf{g}$ be an abelian $G$-subalgebra of $\mf{sl}(V)$.  We will show that $\mf{g}$ is zero.  By
Proposition~\ref{prop:solv} $\mf{g}$ consists of semi-simple elements and is therefore contained in some Cartan
subalgebra of $\mf{sl}(V)$.  This shows that $\dim{\mf{g}} < \dim{V}$.  Thus, by our hypothesis, $\mf{g}=0$.
\end{proof}

Let $V$ be a primitive $G$-module and let $\mf{g}$ be a $G$-subalgebra.  As $\mf{g}$ is semi-simple it decomposes as
$\mf{g}=\bigoplus \mf{g}_i$ where each $\mf{g}_i$ is a simple Lie algebra.  The $\mf{g}_i$ are called the \emph{simple
factors} of $\mf{g}$ and are unique.  As the simple factors are unique, $G$ must permute them.  We call $\mf{g}$
\emph{prime} if the action of $G$ on its simple factors is transitive.  Note that in this case the $\mf{g}_i$'s are
isomorphic and so $\mf{g}$ is ``isotypic.''  Clearly, every $G$-subalgebra of $\mf{sl}(V)$ breaks up into a sum of prime
subalgebras and so it suffices to understand these.

\subsection{The action of a $G$-subalgebra on $V$}

We now consider how a $G$-stable subalgebra acts on $V$:

\begin{proposition}
\label{prop:iso}
Let $V$ be a primitive $G$-module, let $\mf{g}$ be a $G$-subalgebra of $\mf{sl}(V)$ and let $\mf{g}=\bigoplus_{i \in I}
\mf{g}_i$ be the decomposition of $\mf{g}$ into simple factors.
\begin{enumerate}
\item The representation of $\mf{g}$ on $V$ is isotypic, that is, it is of the form $V_0^{\oplus m}$ for some
irreducible $\mf{g}$-module $V_0$.
\item We have a decomposition $V_0=\bigotimes_{i \in I} W_i$ where each $W_i$ is a faithful irreducible representation
of $\mf{g}_i$.
\item We have $V_0 \cong V_0^g$ for each element $g$ of $G$.  (Here $V_0^g$ denotes the $\mf{g}$-module obtained by
twisting $V_0$ by the automorphism $g$ induces on $\mf{g}$.)
\item If $\mf{g}$ is a prime subalgebra then for any $i$ and $j$ one can choose an isomorphism $f:\mf{g}_i \to \mf{g}_j$
so that $W_i$ and $f^* W_j$ become isomorphic as $\mf{g}_i$-modules.
\end{enumerate}
\end{proposition}

\begin{proof}
(1) Since $\mf{g}$ is semi-simple we get a decomposition $V=\bigoplus V_i^{\oplus m_i}$ of $V$ as a $\mf{g}$-module,
where the $V_i$ are pairwise non-isomorphic simple $\mf{g}$-modules.  Each element $g$ of $G$ must take each isotypic
piece $V_i^{\oplus m_i}$ to some other isotypic piece $V_j^{\oplus m_j}$ since the map $g:V \to V^g$ is
$\mf{g}$-equivariant.  As $V$ is primitive for $G$, we conclude that it must be isotypic for $\mf{g}$, and so we may
write $V=V_0^{\oplus m}$ for some irreducible $\mf{g}$-module $V_0$.

(2) As $V_0$ is irreducible, it necessarily decomposes as a tensor product $V_0=\bigotimes_{i \in I} W_i$ where each
$W_i$ is an irreducible $\mf{g}_i$-module.  Since the representation of $\mf{g}$ on $V=V_0^{\oplus m}$ is faithful so
too must be the representation of $\mf{g}$ on $V_0$.  From this, we conclude that each $W_i$ must be a faithful
representation of $\mf{g}_i$.

(3) For any $g \in G$ the map $g:V \to V^g$ is an isomorphism of $\mf{g}$-modules and so $V_0^{\oplus m}$ is isomorphic
to $(V_0^{\oplus m})^g =(V_0^g)^{\oplus m}$, from which it follows that $V_0$ is isomorphic to $V_0^g$.

(4) Since $G$ acts transitively on the simple factors, given $i$ and $j$ we can pick $g \in G$ such that $g \mf{g}_i=
\mf{g}_j$.  The isomorphism of $V_0$ with $V_0^g$ then gives the isomorphism of $W_i$ and $W_j$ as $\mf{g}_i$-modules.
\end{proof}

This proposition gives a strong numerical constraint on prime subalgebras:

\begin{corollary}
\label{cor:prime-constraint}
Let $V$ be a primitive representation of $G$ and let $\mf{g}=\mf{g}_0^n$ be a prime subalgebra of $\mf{sl}(V)$, where
$\mf{g}_0$ is a simple Lie algebra.  Then $\dim{V}$ is divisible by $d^n$ where $d$ is the dimension of some faithful
representation of $\mf{g}_0$.  In particular, $\dim{V} \ge d_0^n$ where $d_0$ is the minimal dimension of a faithful
representation of $\mf{g}_0$.
\end{corollary}

\subsection{Self-dual representations}

Let $V$ be an irreducible self-dual $G$-module.  Thus we have a non-degenerate $G$-invariant form $\langle \cdot ,
\cdot  \rangle:V \otimes V \to \fieldk$.  Such a form is unique up to scaling, and either symmetric or anti-symmetric.
We accordingly call $V$ \emph{orthogonal} or \emph{symplectic}.

Let $A$ be an endomorphism of $V$.  We define the \emph{transpose} of $A$, denoted $A^t$, by the formula
\begin{displaymath}
\langle A^t v, u \rangle=\langle v, Au \rangle.
\end{displaymath}
It is easily verified that $(AB)^t=B^t A^t$ and $({}^g A)^t={}^g (A^t)$.  We call an endomorphism $A$ \emph{symmetric}
if $A=A^t$ and \emph{anti-symmetric} if $A=-A^t$.  One easily verifies that the commutator of two anti-symmetric
endomorphisms is again anti-symmetric.  Thus the set of all anti-symmetric endomorphisms forms a $G$-subalgebra of
$\mf{sl}(V)$ which we denote by $\mf{sl}(V)^-$.  In the orthogonal case $\mf{sl}(V)^-$ is isomorphic to $\mf{so}(V)$ as
a Lie algebra and $\bw{V}$ as a $G$-module, while in the symplectic case it is isomorphic to $\mf{sp}(V)$ as a Lie
algebra and $\Sym^2(V)$ as a $G$-module.  We let $\mf{sl}(V)^+$ denote the space of symmetric endomorphisms.

\begin{proposition}
\label{prop:self-dual}
Let $V$ be an irreducible self-dual $G$-module.  Assume that:
\begin{itemize}
\item $\Sym^2(V)$ and $\bw{V}$ have no isomorphic $G$-submodules; and
\item $\mf{sl}(V)^-$ has no proper non-zero $G$-subalgebras.
\end{itemize}
Then any proper non-zero $G$-subalgebra of $\mf{sl}(V)$ other than $\mf{sl}(V)^-$ is commutative.  In particular, if
$V$ is primitive then the $G$-subalgebras of $\mf{sl}(V)$ are exactly 0, $\mf{sl}(V)^-$ and $\mf{sl}(V)$.
\end{proposition}

\begin{proof}
Let $\mf{g}$ be a non-zero $G$-subalgebra of $\mf{sl}(V)$.  The intersection of $\mf{g}$ with $\mf{sl}(V)^-$ is a
$G$-subalgebra of $\mf{sl}(V)^-$ and therefore either 0 or all of $\mf{sl}(V)^-$.  First assume that the intersection
is zero.  Since the spaces of symmetric and anti-symmetric elements of $\mf{sl}(V)$ have no isomorphic $G$-submodules,
it follows that $\mf{g}$ is contained in the space of symmetric elements of $\mf{sl}(V)$.  But two symmetric elements
bracket to an anti-symmetric element.  Hence all brackets in $\mf{g}$ vanish and so $\mf{g}$ is commutative.  Now
assume that $\mf{g}$ contains all of $\mf{sl}(V)^-$.  It is then a standard fact that $\mf{sl}(V)^-$ is a maximal
subalgebra of $\mf{sl}(V)$ and so $\mf{g}$ is either $\mf{sl}(V)^-$ or $\mf{sl}(V)$.  (To see this, note that
$\mf{sl}(V)=\mf{sl}(V)^- \oplus \mf{sl}(V)^+$ and so to prove the maximality of $\mf{sl}(V)^-$ it suffices to show that
$\mf{sl}(V)^+$ is an irreducible representation of $\mf{sl}(V)^-$.  In the orthogonal case this amounts to the fact
that, as a representation of $\mf{so}(V)$, the space $\Sym^2(V)/W$ is irreducible, where $W$ is the line spanned by the
orthogonal form on $V$.  The symplectic case is similar.)
\end{proof}

\section{The partial derivatives of $\cubic$ have no linear syzygies II}
\label{s:syzygy3}

We now complete the proof of Theorem~\ref{prop:nosyz}.  

\begin{proposition}
\label{prop:subalg}
Assume $\fieldk$ is algebraically closed.  The $\SS_8$-subalgebras of $\mf{sl}(R_1(M_8))$ are exactly 0,
$\mf{so}(R_1(M_8))$ and $\mf{sl}(R_1(M_8))$.
\end{proposition}

\begin{proof}
We begin by noting that any irreducible representation $V$ of any symmetric group $\SS_n$ is defined over $\Q$ and is
therefore orthogonal self-dual.  Thus $\mf{so}(R_1(M_8))=\mf{sl}(R_1(M_8))^-$ makes sense as an $\SS_8$-subalgebra.

For our particular $\SS_8$-representation $R_1(M_8)$, Proposition~\ref{prop:decomp} shows that $\Sym^2(R_1(M_8))$ has
five irreducible submodules of dimensions 1, 14, 14, 20 and 56, while $\bw{R_1(M_8)}$ has two irreducible submodules of
dimensions 35 and 56.  Furthermore, none of these seven irreducible representations are isomorphic.  As all irreducible submodules of
$\mf{sl}(R_1(M_8))$ have dimension at least that of $R_1(M_8)$ (which in this case is 14), we see from
Corollary~\ref{cor:prim-test} that $R_1(M_8)$ is primitive.  (Note that the one-dimensional representation occurring
in $\Sym^2(R_1(M_8))$ is the center of $\mf{gl}(R_1(M_8))$ and does not occur in $\mf{sl}(R_1(M_8))$.)

As $R_1(M_8)$ is primitive, multiplicity free and self-dual, we can apply Proposition~\ref{prop:self-dual}.  This shows
that to prove the present proposition we need only show that $\mf{so}(R_1(M_8))$ has no proper non-zero
$\SS_8$-subalgebras.  Thus assume that $\mf{g}'$ is a proper non-zero $\SS_8$-subalgebra of $\mf{so}(R_1(M_8))$.  As
$\mf{so}(R_1(M_8))=\bw{R_1(M_8)}$ has two irreducible submodules we see that $\mf{g}'$ must be one of these two
irreducible representations.  In particular, this shows that $\mf{g}'$ must be prime and so therefore isotypic.  By examining the list
of simple Lie algebras (see \cite[\S 9.4]{FH}), we see that there are four isotypic semi-simple Lie algebras of
dimension either 35 or 56:
\begin{displaymath}
\mf{g}_2^4, \qquad \mf{so}(8)^2, \qquad \mf{sl}(3)^7, \qquad \mf{sl}(6).
\end{displaymath}
The minimal dimensions of faithful representations of $\mf{g}_2$, $\mf{so}(8)$ and $\mf{sl}(3)$ are 7, 8 and 3.  As
$7^4$, $8^2$ and $3^7$ are all bigger than $\dim{R_1(M_8)}=14$, Corollary~\ref{cor:prime-constraint} rules out the first
three Lie algebras above.  (One can also rule out $\mf{g}_2^4$ and $\mf{sl}(3)^7$ by noting that the alternating group
$A_8$ does not act non-trivially on them.) We rule out $\mf{sl}(6)$ by using Proposition~\ref{prop:iso} and noting that
$\mf{sl}(6)$ has no faithful $14$-dimensional isotypic representation --- this is proved in Lemma~\ref{lem:sl6} below.
(One can also rule out $\mf{sl}(6)$ by noting that $A_8$ does not act non-trivially on it.)  This shows that $\mf{g}'$
cannot exist, and proves the proposition.
\end{proof}

\begin{lemma}
\label{lem:sl6}
The Lie algebra $\mf{sl}(6)$ has exactly two non-trivial irreducible representations of dimension at most $14$:  the
standard representation and its dual.  It has no 14-dimensional faithful isotypic representation.
\end{lemma}

\begin{proof}
For a dominant weight $\lambda$ let $V_{\lambda}$ denote the irreducible representation with highest weight $\lambda$.
If $\lambda$ and $\lambda'$ are two dominant weights then a general fact valid for any semi-simple Lie algebra states
\begin{displaymath}
\dim{V_{\lambda+\lambda'}} \ge \max(\dim{V_{\lambda}}, \dim{V_{\lambda'}}).
\end{displaymath}
(To see this, recall the Weyl dimension formula:
\begin{displaymath}
\dim{V_{\lambda}}=\prod_{\alpha^{\vee}>0} \frac{\langle \lambda+\rho, \alpha^{\vee} \rangle}{\langle \rho,
\alpha^{\vee} \rangle},
\end{displaymath}
where $\rho$ is half the sum of the positive roots and the product is taken over the positive co-roots $\alpha^{\vee}$.
Then note that $\langle \lambda, \alpha^{\vee} \rangle$ is positive for any dominant weight $\lambda$ and any positive
co-root $\alpha^{\vee}$.  Thus $\dim{V_{\lambda+\lambda'}} \ge \dim{V_{\lambda}}$.)

Now, let $\varpi_1, \ldots, \varpi_5$ be the fundamental weights for $\mf{sl}(6)$.  The representation $V_{\varpi_i}$
is just $\bigwedge^i V$, where $V$ is the standard representation.  For $2 \le i \le 4$ the space $V_{\varpi_i}$ has
dimension $\ge 15$.  Furthermore, a simple calculation shows that
\begin{displaymath}
\dim{V_{2 \varpi_1}}=21, \qquad \dim{V_{\varpi_1+\varpi_5}}=168, \qquad \dim{V_{2 \varpi_5}}=21.
\end{displaymath}
(Note that $V_{2 \varpi_1}$ is $\Sym^2(V)$, while $V_{2 \varpi_5}$ is its dual.  This shows why they are 21-dimensional.
To compute the dimension of $V_{\varpi_1+\varpi_5}$ we use the formula for the dimension of the relevant Schur functor,
\cite[Ex.~6.4]{FH}.) Thus only $V_{\varpi_1}$ and $V_{\varpi_5}$ have dimension at most 14, and they each have
dimension 6.  Since 6 does not divide 14 we find that there are no non-trivial 14-dimensional isotypic representations.
\end{proof}

We now have the following:

\begin{proposition}
\label{prop:so-ann}
The only element of $\Sym^3(R_1(M_8))$ annihilated by $\mf{so}(R_1(M_8))$ is zero.
\end{proposition}

\cut{I changed $\cubic$ to $s$ in the following proof, since we're not using \emph{the} cubic.}

\begin{proof}
As mentioned, $R_1(M_8)$ has a non-degenerate symmetric inner product.  Pick an orthonormal basis $\{x_i\}$ of
$R_1(M_8)$ and let $\{x_i^*\}$ be the dual basis of $R_1(M_8)^*$.  We interpret $\Sym^{\bullet}(R_1(M_8))$ as the
polynomial ring in the $x_i$.  The space $\mf{so}(R_1(M_8))$ is spanned by elements of the form $E_{ij}=x_i \otimes
x_j^* - x_j \otimes x_i^*$.  Recall that, for an element $s$ of $\Sym(R_1(M_8))$, the element $x_i \otimes x_j^*$ of
$\End(R_1(M_8))$ takes $s$ to $x_i \frac {\partial s}{\partial x_j}$.  Thus we see that $s$ is annihilated by $E_{ij}$
if and only if 
\begin{equation}
\label{eq1}
x_i \frac {\partial s} {\partial x_j} = x_j \frac {\partial s} {\partial x_i} .
\end{equation}
Therefore $s$ is annihilated by all of $\mf{so}(R_1(M_8))$ if and only if the above equation holds for all $i$ and $j$.

Let $s$ be an element of $\Sym^3(R_1(M_8))$.  We now consider \eqref{eq1} for a fixed $i$ and $j$.  Write
\begin{displaymath}
s=g_3(x_j)+g_2(x_j) x_i+g_1(x_j) x_i^2+g_0(x_j) x_i^3
\end{displaymath}
where each $g_i$ is a polynomial in $x_j$ whose coefficients are polynomials in the $x_k$ with $k \ne i, j$.  Note that
$g_0$ must be a constant by degree considerations.  We have
\begin{displaymath}
\begin{split}
x_i \frac {\partial s} {\partial x_j} &=  g_3'(x_j) x_i + g_2'(x_j) x_i^2 + g_1'(x_j) x_i^3 \\
x_j \frac { \partial s} {\partial x_i} &= x_j g_2(x_j)+2 x_j g_1(x_j) x_i+3 x_j g_0(x_j) x_i^2.
\end{split}
\end{displaymath}
We thus find
\begin{displaymath}
g_2=0, \quad 2x_j g_1=g_3', \quad 3 x_j g_0=g_2', \quad g_1'=0.
\end{displaymath}
{}From this we deduce that $g_0=g_2=0$ and that $g_1$ is determined from $g_3$.  The constraint on $g_3$ is that it
must satisfy
\begin{equation}
\label{eq2}
g_3'(x_j)=x_j g_3''(x_j).
\end{equation}
Putting
\begin{displaymath}
g_3(x_j)=a+b x_j+c x_j^2+d x_j^3
\end{displaymath}
we see that \eqref{eq2} is equivalent to $b=d=0$.  We thus have
\begin{displaymath}
g_3(x_j)=a+c x_j^2, \qquad \textrm{and} \qquad g_1(x_j)=c
\end{displaymath}
and so
\begin{displaymath}
s = a + c (x_i^2+ x_j^2)
\end{displaymath}
is the general solution to \eqref{eq1}.

We thus see that if $s$ satisfies \eqref{eq1} for a particular $i$ and $j$ then $x_i$ and $x_j$ occur in $s$ with only
even powers.  Thus if $s$ satisfies \eqref{eq1} for all $i$ and $j$ then all variables appear to an even power.  This
is impossible, unless $s=0$, since $s$ has degree three.  Thus we see that zero is the only solution to \eqref{eq1}
which holds for all $i$ and $j$.
\end{proof}

\begin{remark}
The above computational proof can be made more conceptual.  By considering the equation \eqref{eq1} for a fixed $i$ and
$j$ we are considering the invariants of $\Sym^3(R_1(M_8))$ under a certain copy of $\mf{so}(2)$ sitting inside of
$\mf{so}(R_1(M_8))$.  The representation $R_1(M_8)$ restricted to $\mf{so}(2)$ decomposes as $S \oplus T$ where $S$ is
the standard representation of $\mf{so}(2)$ and $T$ is a 12-dimensional trivial representation of $\mf{so}(2)$.  We
then have
\begin{displaymath}
\Sym^3(R_1(M_8))^{\mf{so}(2)}=\bigoplus_{i=0}^3 \Sym^i(S)^{\mf{so}(2)} \otimes \Sym^{3-i}(T).
\end{displaymath}
Finally, our general solution to \eqref{eq1} amounts to the fact that the ring of invariants
$\Sym^{\bullet}(S)^{\mf{so}(2)}$ is generated by the norm form $x_i^2+x_j^2$.
\end{remark}

We can now prove Proposition~\ref{prop:l-zero}, which will establish Theorem~\ref{prop:nosyz}.

\begin{proof}[Proof of Proposition~\ref{prop:l-zero}]
To prove $\mf{g}=0$ we may pass to the algebraic closure of $\fieldk$; we thus assume $\fieldk$ is algebraically closed.
By Proposition~\ref{prop:subalg}, the Lie algebra $\mf{g}$ must be 0, $\mf{so}(R_1(M_8))$ or $\mf{sl}(R_1(M_8))$.  By
Proposition~\ref{prop:so-ann}, $\mf{g}$ cannot be $\mf{so}(R_1(M_8))$ or $\mf{sl}(R_1(M_8))$ since it annihilates
$\cubic$, and $\cubic$ is non-zero.  Thus $\mf{g}=0$.
\end{proof}

\section{The minimal graded free resolution of the graded ring of $M_8$}
\label{s:freeresolution}

We now determine the minimal graded free resolution of the invariant ring $R_{\bullet}(M_8)$.  We first review some
commutative algebra.  

\subsection{Betti numbers of modules over polynomial rings}

Let $P_{\bullet}$ be a graded polynomial ring over $\fieldk$ in finitely many indeterminates, each of positive degree.
Let $M$ be a finitely generated graded $P_{\bullet}$-module.  (To follow our convention of keeping track of graded
objects, we should write $M_{\bullet}$ rather than $M$.  But because we will be resolving $M$, we do not.)
\cut{I apologize that I have written $R_{\bullet}$.  It is awkward in this section.}

One can then find a surjection $F \to M$ with $F$ a finite free module having the following property: if $F' \to M$ is
another surjection from a finite free module then there is a surjection $F' \to F$ making the obvious diagram commute.  
This \emph{free envelope} $F \to M$ of $M$ is unique up to non-unique isomorphism.  

Build a resolution of $M$ by using free envelopes:
\begin{displaymath}
\cdots \to F_2 \to F_1 \to F_0 \to M \to 0
\end{displaymath}
Here $F_0$ is the free envelope of $M$ and $F_{i+1}$ is the free envelope of $\ker(F_i \to F_{i-1})$.  Define integers
$b_{i, j}$ by
\begin{displaymath}
F_i=\bigoplus_{j \in \Z} P_{\bullet}[-i-j]^{\oplus b_{i, j}}.
\end{displaymath}
The $b_{ij}$ are called the \emph{Betti numbers} of $M$ and collectively they form  the \emph{Betti diagram} of $M$.
They are independent of the choice of free envelopes, as $b_{i, j}$ is also the dimension of the $j$th graded piece of
$\Tor_i^{P_{\bullet}}(M, P_{\bullet}/I_{\bullet})$, where $I_{\bullet}$ is ideal of positive degree elements.  The Betti
numbers have the following properties:
\begin{itemize}
\item[(B1)] We have $b_{i, j}=0$ for all but finitely many $i$ and $j$.  This follows since each $F_i$ is finitely
generated and $F_i=0$ for $i$ large by Hilbert's syzygy theorem.
\item[(B2)] We have $b_{i, j}=0$ for $i<0$.  This follows from the definition.
\item[(B3)] If $b_{i_0, j}=0$ for $j \le j_0$ then $b_{i, j}=0$ for all $i \ge i_0$ and $j \le j_0$.  This follows from
the fact that if $d$ is the lowest degree occurring in a module $M$ and $F \to M$ is a free envelope then $F_d \to M_d$
is an isomorphism, and thus the lowest degree occurring in $\ker(F \to M)$ is at least $d+1$.
\item[(B4)] In particular, if $M$ is supported in non-negative degrees then $b_{i, j}=0$ for $j<0$.
\item[(B5)] Let $f(k)=\dim{M_k}$ (resp.\ $g(k)=\dim{P_k}$) denote the Hilbert function of $M$ (resp.\ $P_{\bullet}$).
Then
\begin{displaymath}
f(k)=\sum_{i, j \in \Z} (-1)^i \cdot b_{i, j} \cdot g(k-i-j).
\end{displaymath}
This follows by taking the Euler characteristic of the $k$th graded piece of $\cdots \to F_1 \to F_0 \to M$.
\end{itemize}
In particular,  if $M$ is supported in non-negative degrees, then its Betti diagram is contained in a bounded subset of
the first quadrant.

\subsection{Betti numbers of graded algebras}

Let $R_{\bullet}$ be a finitely generated graded $k$-algebra,  generated in degree one.  Let $P_{\bullet}=
\Sym^{\bullet}(R_1)$ be the graded polynomial algebra on the first graded piece, so $R_{\bullet}$ is a
$P_{\bullet}$-module, and we can speak of its Betti numbers of $R_{\bullet}$ (as a $P_{\bullet}$-module).

Assume now that the ring $R_{\bullet}$ is Gorenstein and a domain.  The canonical module $\omega_R$ of $R_{\bullet}$ is
then naturally a graded module.  Furthermore, there exists an integer $a$, called the \emph{$a$-invariant} of
$R_{\bullet}$, such that $\omega_R$ is isomorphic to $R_{\bullet}[a]$.  We now have the following important property of
the Betti numbers of $R_{\bullet}$:
\begin{itemize}
\item[(B6)] We have $b_{i, j}=b_{r-i, d+a-j}$ where $d=\dim{R_{\bullet}}$,
\begin{displaymath}
r=\dim{P_{\bullet}}-\dim{R_{\bullet}}=\codim(\Spec(R_{\bullet}) \subset \Spec(P_{\bullet})),
\end{displaymath}
and $a$ is the $a$-invariant of $R_{\bullet}$.
\end{itemize}
No doubt this formula appears in the literature, but we derive it here for completeness.  We have $\Ext^i_{P_{\bullet}}
(R_{\bullet}, \omega_{P_{\bullet}}) \cong \omega_R$ if $i=r$ and $0$ if $i \neq r$.  If $n$ is the dimension of
$P_{\bullet}$, then $\omega_{P_{\bullet}} \cong P_{\bullet}[-n]$.  Since $R_{\bullet}$ is Gorenstein we have $\omega_R
\cong R_{\bullet}[a]$.  Therefore we obtain a minimal free resolution $\cdots \to G_1 \to G_0 \to R_{\bullet}[a] \to 0$
of $R_{\bullet}[a]$ by $G_i=\Hom_{P_{\bullet}}(F_{r-i}, P_{\bullet}[-n])$.  Then $\cdots \to G_1[-a] \to G_0[-a] \to
R_{\bullet} \to 0$ is a minimal free resolution of $R_{\bullet}$, and by uniqueness of the resolution we therefore have
$G_i[-a] \cong F_i$ for each $i$.  Now $G_i[-a] \cong \bigoplus_{j'} P_{\bullet}[-n-r+i+j'-a]$, and so
\begin{displaymath}
\bigoplus_{j'} P_{\bullet}[-n+r-i+j'-a]^{b_{r-i,j'}} \cong \bigoplus_j P_{\bullet}[-i-j]^{b_{i,j}}.
\end{displaymath}
Equating components of the same degree gives $-n+r-i+j'-a=-i-j$, or $j'=n-r+a-j$.  Hence $b_{i,j}=b_{r-i, n-r+a-j}
=b_{r-i, d+a-j}$.

\subsection{The minimal graded free resolution of  $R_{\bullet}(M_8)$}
\label{s:betti}

We begin with the following result:

\begin{proposition}
\label{prop:goren}
The ring $R_{\bullet}(M_8)$ is Gorenstein with $a$-invariant $-2$.
\end{proposition}

\begin{proof}
We recall a theorem of Hochster--Roberts \cite[Theorem~6.5.1]{BrunsHerzog}:  if $V$ is a representation of the reductive
group $G$ (over a field $\fieldk$ of characteristic zero) then the ring of invariants $(\Sym^{\bullet}{V})^G$ is
Cohen--Macaulay.  As our ring $R_{\bullet}(M_8)$ can be realized in this manner, with $V$ being the space of
$2 \times 8$ matrices and $G=\SL(2) \times T$, where $T$ is the maximal torus in $\SL(8)$, we see that
$R_{\bullet}(M_8)$ is Cohen--Macaulay.  We next recall a theorem of Stanley \cite[Corollary~4.4.6]{BrunsHerzog}:  if
$R_{\bullet}$ is a Cohen--Macaulay ring generated in degree one with Hilbert series $f(t)/(1-t)^d$, where $d$ is the
Krull dimension of $R_{\bullet}$, then $R_{\bullet}$ is Gorenstein if and only if the polynomial $f$ is symmetric.  In
this case, the $a$-invariant of $R_{\bullet}$ is $\deg{f}-d$.  Going back to our situation, the Hilbert series of our
ring is given in \eqref{howenowe}.  The numerator is symmetric of degree four and the denominator has degree six.  Thus
$R_{\bullet(M_8)}$ is Gorenstein with $a=-2$.
\end{proof}

We can now deduce the Betti diagram of $R_{\bullet}(M_8)$:

\begin{proposition}
\label{prop:betti}
The Betti diagram of $R_{\bullet}(M_8)$ is given by:
\vskip 2ex
\begin{center} \rm
\begin{tabular}{|c||c|c|c|c|c|c|c|c|c|}
\hline
$j \bs i$ & \hskip .2em 0 \hskip .2em & 1 & 2 & 3 & 4 & 5 & 6 & 7 &
\hskip .2em 8 \hskip .2em \\
\hline
\hline
0 & 1 & 0 & 0 & 0 & 0 & 0 & 0 & 0 & 0 \\
\hline
1 & 0 & 14 & 0 & 0 & 0 & 0 & 0 & 0 & 0 \\
\hline
2 & 0 & 0 & 175 & 512 & 700 & 512 & 175 & 0 & 0 \\
\hline
3 & 0 & 0 & 0 & 0 & 0 & 0 & 0 & 14 & 0 \\
\hline
4 & 0 & 0 & 0 & 0 & 0 & 0 & 0 & 0 & 1 \\
\hline
\end{tabular}
\end{center}
\vskip 2ex
All $b_{i,j}$ not shown  are zero.
\end{proposition}

\begin{proof}
We first note that (B6) gives $b_{8-i, 4-j}=b_{i, j}$ as $r=8$, $d=6$ and $a=-2$.  We thus have the symmetry of the
table. Now, by (B2) and (B4), $b_{i,j}=0$ if either $i$ or $j$ is negative.  Thus  $b_{i,j}=0$ if $i>8$ or $j>4$ by
symmetry.  Next, observe that $P_{\bullet} \to R_{\bullet}(M_8)$ is the free envelope of $R_{\bullet}(M_8)$, where
$P_{\bullet}=\Sym^{\bullet}(R_1(M_8))$.  This gives the $i=0$ column of the table.  We now look at the $i=1$ column.
The 14 generators have no linear relations, so $b_{1,0}=0$.  By (B3), $b_{i,0} =0$ for $i \ge 1$.  We also know that
there are 14 quadratic relations, so $b_{1,1}=14$.  We now look at the $i=2$ column of the table.  The 14 quadratic
relations have no linear syzygies (Theorem~\ref{prop:nosyz}), so $b_{2,1}=0$.  Using (B3) again, we conclude $b_{i,1}=0$
for $i \ge 2$.  We have thus completed the first two rows of the table.  The last two rows are then determined by
symmetry.  The middle row can now be determined from (B5) by evaluating both sides at $k=2, \ldots, 10$ and solving the
resulting upper triangular system of equations for $b_{i,2}$.  (In fact, the computation is simpler than that since
$b_{i,2}=b_{8-i,2}$ and we know $b_{0,2}=b_{1,2}=0$, the latter vanishing coming from Corollary~\ref{prop:i2-i3}.)
\end{proof}

\begin{remark}
\label{r:denver}
Proposition~\ref{prop:betti} (in particular, the $i=1$ column of the table) shows that $I_{\bullet}(M_8)$ is generated
by its degree two piece.  Thus we have another proof of Corollary~\ref{generation}.
\end{remark}

\begin{remark}
The resolution of $R_{\bullet}(M_8)$ as a $P_{\bullet}$-module, without any consideration of grading, is given by
Freitag and Salvati Manni \cite[Lemma~1.3, Theorem~ 1.5]{fs}.  It was obtained by computer.
\end{remark}

\section{The degree one and two invariants of $N_8$: generation of
  $R_{\bullet}(N_8)$, and representation theory}
\label{s:N8}

In this section, we determine part of the structure of the ring
$R_{\bullet}(N_8)$ of invariants of $8$ points in $\proj^3$ with the
assistance of a computer.  In particular, we show that this ring is
generated in degree $1$ and $2$, and we
determine the actions of Gale duality and $\SS_8$ on the generators.
As a consequence, we show that the Gale-invariant invariants are
precisely the subring of $R_{\bullet}(N_8)$ generated in degree $1$,
and the skew quintic $\quintic$ is the unique skew quintic relation in both
$R_{\bullet}(N_8)$ and $R_{\bullet}(N'_8)$.

\begin{proposition}
The ring of invariants $R_{\bullet}(N_8)$ is generated in degrees
  one and two.\label{browncow}
\end{proposition}

\begin{proof}
A filtration of the ring of invariants such that the associated graded ring is the semigroup of Gel'fand-Tsetlin patterns (or
equivalently,
semistandard tableaux), as described for example by Alexeev and Brion
in \cite[\S 5.1]{ab}\cut{in Gonciulea-Lakshmibai, Caldero,
Alexeev-Brion, or Foth-Hu}, can be used to show that the ring is
generated in degrees at most $4$.  (Code using the package 4ti2 is available at
\cite{hmsvcode},
 but the method is standard.)  

To show that the ring is  generated in degrees $1$ and $2$ is
then just linear algebra.  We compute the dimensions of the subspaces
of the degree $3$ and $4$ pieces generated by the degree $1$ and $2$ tableaux.  
Magma code is available 
at  \cite{hmsvcode}.
(In more detail: We use a Grobner basis
for the ideal of Pl\"ucker relations.  We define a polynomial ring in 70
variables corresponding to the $\binom 8 4$ minors of a $4 \times 8$
matrix.  In a certain term order, the Pl\"ucker relations are a Grobner
basis for the Pl\"ucker ideal, see \cite[Thm.~14.6, p.~277]{ms}.  We define the monomials corresponding to the Hilbert basis
output from the previous 4ti2 program.   The point of using the
Grobner basis is that one can quickly implement the straightening
relations.  The normal form of any polynomial will have all
semistandard tableaux as monomials.   Now, one simply multiplies degree
$1$ and degree $2$ tableaux and then computes the dimension of degree
$3$, resp.\ degree $4$, subspaces spanned by them.)
\end{proof}

By Corollary~\ref{c:united},  there are no degree $2$ relations
for $N'_8$, so $\Sym^2(R_1(N_8)) \rightarrow R_2(N_8)$ is an injection.

\begin{proposition} $\quad$
\begin{enumerate}
\item[(a)] Gale duality acts via the trivial representation on $R_1(N_8)$.
\item[(b)] Gale duality acts via the sign representation on $R_2(N_8) /
\Sym^2 R_1(N_8)$.
\end{enumerate}
Thus by Proposition~\ref{browncow}, $N'_8 := \Proj \Sym^{\bullet} (R_1(N_8))$ is the quotient of $N_8$ by Gale-duality.\label{anotherbrowncow}
\end{proposition}

\begin{proof}
By \cite[Thm.~1.12, p.~690]{hm}, 
Gale-duality acts as follows.   The column of a tableaux, with rows
$\text{abcd}$, is replaced by $\text{efgh}$, where $\{ \text{a}, \text{b}, \text{c}, \dots, \text{h} \} = \{ 1,
\dots, 8 \}$, and $\text{a}< \text{b} < \text{c} < \text{d}$ and
$\text{e}< \text{f} < \text{g} < \text{h}$, with a sign factor of
$\sgn( \text{abcdefgh})$ (where ``$\text{abcdefgh}$'' is interpreted as
an element of $\SS_8$, i.e.\ $1 \mapsto a$, $2 \mapsto b$, etc.).  

(a)   A degree one tableau (a generator of $R_1(N_8)$) has two columns
$\text{abcd}$ and $\text{efgh}$, which are swapped by this process.  As
$\sgn(\text{abcdefgh}) = \sgn(\text{efghabcd})$, the sign contributions cancel.
Thus every degree one tableau is Gale-invariant.

(b)  
By
enumerating
semistandard tableaux, we see that $$\dim 
R_2(N_8) /
\Sym^2 R_1(N_8) = 21.$$  One readily sees that there are $42$ degree
$2$ tableaux  not fixed by the Gale-involution.  They come in
$21$ pairs, and the differences of elements of each pair form a base
for the $(-1)$-eigenspace of the Gale involution.
\end{proof}

\cut{Ravi says:  I'd asked (Feb. 12 2010):  if we knew a little more, e.g. no linear
syzygies among the quartics, would that be enough to know the entire
free resolution of $N'_8$?  No one had any ideas.
Ben knows no degree $2$ relations, see August 11, 2008 and Oct. 23,
2008.  
Also no degree $3$ relations, 10/23/08.  I think I can do both via Bezout.  He knows the relations are
in degree at most $12$, 11/3/08.  I'm not sure if that includes the
degree $2$ generators.

Andrew says: 
I think I may be able to prove that the quartics have no linear syzygies, as follows.  Just as with
the linear syzygies of the derivatives of $\cubic$, the linear syzygies of the derivatives of $\quintic$ form
an $\SS_8$-stable Lie subalgebra of $\End(R_1(N_8))$.  As $\SS_8$ Lie algebras, $\End(R_1(N_8)) \cong \End(R_1(M_8))$,
and so our classification of $\SS_8$-subalgebras of the latter gives a classification of $\SS_8$-subalgebras of the
former.  As before, I'm guessing that $\mf{so}(R_1(N_8))$ cannot annihilate $\quintic$ since $\quintic$ has odd
degree.  Thus the annihilator of $\quintic$ is 0, i.e., there are no linear syzygies of its derivatives.}

\subsection{The skew quintic $\quintic$ is the unique skew quintic
  relation} 

We can now readily compute the representation of $\SS_8$ on $R_2(N_8) /
\Sym^2(R_1(N_8))$, and show that it is isomorphic to the irreducible
representation $V_{3,1,1,1,1,1}$.  To do this, we analyze $R_2(N_8)$
using  Schur-Weyl duality,
$$\oplus_{\lambda \vdash 8} S_\lambda( \Sym^2(C^4)) \otimes V_\lambda,$$ 
and examine  the case $\lambda = (3,1,1,1,1,1)$.
We seek the dimension of the $SL_4$-invariant part of 
$$S_{3,1,1,1,1,1}(S_{2}(\C^4)) = \oplus_{\mu \vdash 16} n_\mu S_\mu(\C^4),$$ 
where $n_\mu$ is the multiplicity of the Schur functor $S_\mu$.  
The dimension of the $SL_4$-invariant part is equal to $n_{4,4,4,4}$, 
which is the number of copies of $S_{4,4,4,4}$ within the plethysm 
$S_{3,1,1,1,1,1}(S_{2}( - ))$.  This can be checked in any number of
algebra packages.  For example,  a calculation in  Maple  is given in
\cite{hmsvcode}.

\begin{proposition}
The skew quintic $\quintic$ is the unique skew quintic relation in
$N_8$, and hence in $N'_8$.\label{uniqueskewquintic}
\end{proposition}

This is now a straightforward verification.  As observed in the proof
of Proposition~\ref{p:repfacts}(d), the sign representation appears
with multiplicity $4$ in $\Sym^5(R_1(N_8))$, and with multiplicity $3$
in $R_5(N_8)$. 
By Theorem~\ref{t:Qdual}, $\quintic \in \ker( \Sym^5(R_1(N_8)) \rightarrow
R_5(N_8))$.
Let $W$ be an $\SS_8$-equivariant lift of $R_2(N_8) /
\Sym^2(R_1(N_8))$ to $R_2(N_8)$, so $R_2(N_8) = \Sym^2(R_1(N_8))
\oplus W$ as $\SS_8$ representations, and $W \cong V_{3,1,1,1,1,1}$.
By the generation of $R_{\bullet}(N_8)$  in degrees up to two
(Proposition~\ref{browncow}), we have a {\em surjection}
$$
\Sym^5(R_1(N_8)) \oplus W \otimes \Sym^3 (R_1(N_8)) \oplus \left(
  \Sym^2 W \right)
\otimes R_1(N_8) \rightarrow R_5(N_8).$$
The result then follows by checking that the sign representation does
not appear in 
$W \otimes \Sym^3 (R_1(N_8))$ or $(  \Sym^2 W )
\otimes R_1(N_8)$, which may be verified using character theory.
(See \cite{hmsvcode} for maple code.)

\bigskip

{\tiny

Benjamin Howard:
Center for Communications Research,
Princeton, NJ 08540, bjhowa3@idaccr.org

\smallskip

John Millson:
Department of Mathematics,
University of Maryland,
College Park, MD 20742, USA,
jjm@math.umd.edu

\smallskip

Andrew Snowden:
Department of Mathematics,
MIT, Cambridge, MA 02139,
asnowden@math.mit.edu

\smallskip

Ravi Vakil:
Department of Mathematics,
Stanford University,
Stanford, CA 94305, USA,
vakil@math.stanford.edu

}

\end{document}